\documentclass[12pt]{amsart}
\usepackage[margin=1.25in]{geometry}
\usepackage[utf8]{inputenc}
\usepackage{xcolor}
\usepackage{hyperref}
\usepackage{multirow}
\usepackage{booktabs} 
\usepackage{array}    

\usepackage{amssymb}

\makeatletter
\newcommand\reallytiny{\@setfontsize\reallytiny{5}{6}}
\makeatother
\newcommand{\unit}{\text{\textbf{1}}}
\newcommand{\cR}{\mathcal{R}}

\newcommand{\cC}{{\mathcal C}}

\newcommand{\Z}{{\mathbb Z}}

\newcommand{\C}{{\mathbb{C} }}

\newcommand{\bZ}{\mathbb{Z}}

\newcommand{\id}{\operatorname{id}}

\newcommand{\cB}{\mathcal{B}}

\newcommand{\Fun}{\operatorname{Fun}}

\newcommand{\Rep}{\operatorname{Rep}}
\newcommand{\Irr}{\operatorname{Irr}}

\newcommand{\Hom}{\operatorname{Hom}}

\newcommand{\Id}{\operatorname{Id}}

\newcommand{\mfg}{\mathfrak{g}}

\theoremstyle{plain}

\numberwithin{equation}{section}

\newtheorem{theorem}{Theorem}[section]
\newtheorem{lemma}[theorem]{Lemma}

\newtheorem{corollary}[theorem]{Corollary}
\newtheorem{proposition}[theorem]{Proposition}

\newcommand{\xdownarrow}[1]{%
  {\left\downarrow\vbox to #1{}\right.\kern-\nulldelimiterspace}
}

\theoremstyle{definition}
\newtheorem{definition}[theorem]{Definition}
\newtheorem{example}[theorem]{Example}

\theoremstyle{remark}

\newtheorem{remark}[theorem]{Remark}

\author[C. Galindo]{C\'esar Galindo}
\address{ Departamento de Matem\'aticas, Universidad e los Andes, Bogot\'a, Colombia}
\email{cn.galindo1116@uniandes.edu.co}
\author[G. Mora]{Giovanny Mora}
\address{ Departamento de Matem\'aticas, Universidad de los Andes, Bogot\'a, Colombia}
\email{hg.mora@uniandes.edu.co}
\author[E. Rowell]{Eric C. Rowell}
\address{ Department of Mathematics, Texas A\&M University, College Station, TX, USA}
\email{rowell@tamu.edu}
\begin{document}

\title[Braided Zestings of Verlinde Modular Categories]{Braided Zestings of Verlinde Modular Categories and Their Modular Data}

\thanks{C.G. was partially supported by Grant INV-2023-162-2830 from the School of Science of Universidad de los Andes. G. M. was partially supported by Grant INV-2023-175-2903 from the School of Science of Universidad de los Andes. E.C.R. was partially supported by US NSF grant DMS-2205962.}

\begin{abstract}
Zesting of braided fusion categories is a procedure that can be used to obtain new modular categories from a modular category with non-trivial invertible objects. In this paper, we classify and construct all possible braided zesting data for modular categories associated with quantum groups at roots of unity. We produce closed formulas, based on the root system of the associated Lie algebra, for the modular data of these new modular categories.
\end{abstract}


\date{\today}
\maketitle

\section{Introduction}
The representation categories $\Rep(U_q \mathfrak{g})$ of quantum groups at roots of unity are a ubiquitous source of modular categories, obtained as certain subquotients (see, e.g., \cite{BK}). In \cite[Section 8.18.2]{EGNO} these are called \emph{Verlinde categories}, which we denote by $\cC(\mfg,k)$: here $k\in\bZ^{>0}$ is the level which depends only on the order of the root of unity, see Section \ref{section:Verlinde}. In low ranks, where classifications are available, nearly all modular categories are subcategories of some Verlinde category or products of them: see \cite{RSW,NRWW} for ranks$\leq 6$, from which one can see that this is strictly the case at the level of fusion rules. For modestly larger ranks one starts to see modular data without such realizations \cite{NRW}, and it becomes a challenge to find any realization.

\emph{(Braided) zesting} \cite{zesting} is a construction technique developed to realize fusion rules that are closely related to those of known categories, encountered in various classifications \cite{AIM2012,NRWW}. 
Zesting can be applied to any braided fusion category $\cB$ with a $G$-grading and some non-trivial invertible objects in the trivially graded component.  The main idea is to twist the fusion rules of $\cB$ by means of a 2-cocycle $\lambda$ on the grading group with coefficients in the invertible objects in the the trivial component.  One then adjusts the associativity constraint appropriately (if possible) to obtain a new fusion category $\cB^\lambda$.  Then one can look for suitable modifications of the original braiding on $\cB$ to equip $\cB^\lambda$ with a braiding.

If $\mathcal{B}$ is pseudo-unitary, modular, and the grading group $A$ is the universal one, it can be shown \cite{zesting} that any braided zesting has a canonical modular structure that is also pseudo-unitary. In this case, we can shift our focus to $\mathcal{B}_{\text{pt}}$, the maximal pointed part of $\mathcal{B}$, which is completely determined by the modular data $M$. We can then talk about a zesting of the pre-metric group associated with $\mathcal{B}_{\text{pt}}$ or the modular data itself; see Theorem \ref{th:new_modular_data}.
For a pre-metric group $(A, q)$ (see Definition \ref{def: premetirc group A0 y c}), we define the set of equivalence classes $\mathbf{Sez}(A, q)$ of zesting data for $(A, q)$ and provide a cohomological parametrization of inequivalent zesting data, and similarly for the braided zesting data. This facilitates a thorough analysis of zestings of the Verlinde categories.

Modular Verlinde categories $\cC(\mfg,k)$ often have a $A$-grading and thus braided zesting can be applied. Moreover, these are typically pseudo-unitary. While several examples have been carried out in \cite{zesting}, most cases remain unexplored.  The universal grading group of $\cC(\mfg,k)$ is always cyclic, except for type $D_r$ with $r$ even in which case the grading group is the Klein four group $\bZ/2\times\bZ/2$.  The methods of \cite{zesting} can be employed to handle the cyclic cases, but the complicated nature of the cohomology of the Klein 4 group offers a new challenge.  In particular there are two distinct cases, depending on if the invertible objects are all bosons or if there are fermions present.  

In this article we work out the modular data of braided zestings for all pseudo-unitary modular Verlinde categories, in each case taking the universal grading.  This is achieved by developing methods for handling the non-cyclic cases and adapting known methods for the cyclic cases.  We provide several interesting new examples of categories not directly constructed as subcategories of $\mathcal{C}(\mathfrak{g}, k)$; see Examples~\ref{example:E6level3}, \ref{example:so7level2}, and \ref{example:sp6level2}, and Section~\ref{subsubsection:so4s}.

Here is a more detailed outline of the contents of this paper. In Section~\ref{section:prelim}, we recall basic definitions, and in Section~\ref{section:zesting}, we describe the zesting construction and provide the necessary cohomological notation. Some specific groups are treated in Section~\ref{section:families}, with their cohomological data carefully parameterized. In Section~\ref{section:Verlinde}, we introduce the Verlinde categories, and in Subsection~\ref{ExamplesVerlinde}, we explicitly compute the zestings of their modular data. Additional technical details and notation are provided in the appendix.

\section{Preliminaries}\label{section:prelim}

We will briefly recall the basic definitions of fusion ring and modular data. For details of the results cited in this section, we refer the reader to \cite{EGNO} and \cite{BK}.

\subsection{Fusion categories and fusion rings}
By a fusion category, we mean a $\mathbb{C}$-linear, semisimple, rigid tensor category, $\mathcal{C}$, with a finite number of isomorphism classes of simple objects and finite-dimensional spaces of morphisms. Furthermore, we require that $\unit$, the unit object of $\cC$, is a simple object. By a fusion subcategory of a fusion category $\cC$, we always mean a full tensor subcategory. An example of a subcategory is the maximal pointed subcategory $\cC_{\text{pt}} \subset \cC$ generated by the invertible objects of $\cC$. We say that $\cC$ is pointed if $\cC = \cC_{\text{pt}}$. We refer the reader to \cite{EGNO} for a comprehensive theory of such categories.

Let $\cC$ be a fusion category. Its Grothendieck ring, $K_0(\cC)$, is the free $\mathbb{Z}$-module generated by $\Irr(\cC)=\{X_0=\unit, X_1,\ldots,X_n\}$ 
(the set of isomorphism classes of simple objects of $\cC$) with multiplication induced from the tensor product in $\cC$. The Grothendieck ring of a fusion category exhibits a combinatorial structure known as a \emph{fusion ring}, which we will recall in the following.

A \emph{fusion ring} $\cR$ is a finitely generated ring with a unit $1$, free as a $\mathbb{Z}$-module, and equipped with a distinguished basis $B=\{ b_i \}_{i\in I}$ where $b_0=1\in B$. This fulfills the condition:
\[b_i b_j = \sum_{k\in I}N_{ij}^k b_k\]
where $N{ij}^k \in \mathbb{Z}^{\geq 0}$ are called the structure constants of $R$, and an involution $*: B \to B$ such that the induced map 

\begin{align*}
    *: R&\to R, \quad
    r=\sum_{i\in I}r_i b_i \mapsto r^*= \sum_{i\in I}r_i b_i^* , \ \ \ r_i\in \mathbb{Z},
\end{align*}
is an anti-involution of the ring $\cR$ and
$$N_{i,j}^0= \begin{cases}
1, & \text{ if } i=j^* \\
0, & \text{ else.}
\end{cases}
$$

The prototypical example of a fusion ring is $K_0(\mathcal{C})$, the Grothendieck ring of a fusion category $\mathcal{C}$. Here,  the basis consists of the set of isomorphism classes of 
simple objects $\operatorname{Irr}(\mathcal{C})$. For each $ X_i \in \mathrm{Irr}(\mathcal{C})$ (where $\mathrm{Irr}(\mathcal{C})$ denotes the set of isomorphism classes of simple objects), define $i^*$ as an element in the index set $I = \{0, 1, \ldots, n\}$ such that $X_{i^*}$ is defined to be $X_i^*$.
Furthermore, for $X_i, X_j, X_k \in \operatorname{Irr}(\mathcal{C})$, define $N_{i,j}^k := \operatorname{dim}_{\mathbb{C}}\operatorname{Hom}_{\mathcal{C}}(X_k, X_i \otimes X_j)$. For additional details, refer to \cite{EGNO}.

\subsection{Modular categories and their modular data}
We provide the essential definitions and notation for modular categories. For an exhaustive discussion encompassing all these axioms and definitions, consult \cite{EGNO}.
A \emph{braiding} in a fusion category is a natural family of isomorphisms $c_{X,Y}:X\otimes Y\cong Y\otimes X$ satisfying the hexagon axioms. A fusion category equipped with a braiding is known as a \emph{braided fusion category}.

In a braided fusion category, a \emph{ribbon twist} is a natural family of isomorphisms $\theta_X:X\to X$ that is compatible with the braiding and duality. This compatibility permits the definition of a \emph{quantum trace}, denoted by $\operatorname{Tr}:\operatorname{End}\cC(X)\to \mathbb{C}, f\mapsto \operatorname{Tr}(f)$.

A \emph{premodular category} is a braided fusion category with a twist. It gains the label \emph{modular} if the \emph{$S$-matrix}, defined by $s_{i,j}:=\operatorname{Tr}(c_{X_i,X_j}\circ c_{X_j,X_i})$ where $X_i,X_j \in \Irr(\cC)=\{X_0=\unit,X_1,\ldots, X_n\}$, is invertible.

Each dimension $d_i := s_{0,i}=\operatorname{Tr}(\operatorname{id}_{X_i})$ is a real number. The entries of the $S$-matrix are known to satisfy $s_{ij} = s_{ji} = s_{i^*j^*}$ and  $\sum_{j} s_{ij} s_{jl} = \delta_{i^*l}D^2$, where $D = \sqrt{\sum_{i \in I} d_i^2}$ is the \emph{categorical dimension} of $\cC$.

Since the twist of a simple object is a scalar times the identity, for any simple object, we will identify the twist $\theta_i\Id_{X_i}:=\theta_{X_i}$ with the complex number $\theta_i$. Each $\theta_i$, $i \in I$, is a root of unity. The diagonal matrix defined by $T_{ij}=\delta_{ij}\theta_i^{-1}$ is called the $T$-matrix. 

Of all the identities involving the $S$-matrix and the $T$-matrix, the \emph{Verlinde formula} stands out as particularly significant. This equation implies that the structure constants of the Grothendieck rings are entirely determined by the $S$-matrix.

\begin{equation*}\label{eq:verlinde_formula}
N_{j,k}^{l}=\frac{1}{D}\sum_{i\in I}\frac{s_{i,j}s_{i,k}s_{i^*,l}}{d_i}, \quad i,j,k\in I.
\end{equation*}

The $S$-matrix and the $T$-matrix exhibit numerous properties, which are concisely captured in the notion of a \emph{modular datum}.
 This concept acts as a combinatorial parallel to the modular category, analogous to how a fusion ring is intrinsically connected to a fusion category. Such definitions are important as modular categories often are classified based on their modular data, as seen in \cite{RSW,NRWW,NRW}.

\begin{definition}
A \emph{modular datum} $M$ consists of a finite set $I$, with its elements called \emph{labels} of $M$, including  a unique element $0\in I$, and complex $|I|$-square matrices, $S$ and $T$, named the $S$-matrix and $T$-matrix respectively. The matrices must satisfy:
\begin{enumerate}
    \item $T$ is diagonal, $S$ is symmetric, and the scalars $d_i:=s_{0,i} \neq 0$ for all $i\in I$ with $s_{0,0}=1$.
    \item The functions $b_i (j):=\frac{s_{ij}}{s_{0j}}$ in $\Fun(I, \mathbb{C})$ span a fusion ring $\cR_{M}$, where they form the distinguished basis.
    \item $d_i=d_{i^*}$.
    \item The matrices $S$ and $T$ yield a projective representation of the modular group $SL_{2}(\mathbb{Z})$, implying the existence of a $\tau \in \mathbb{C}^*$ such that $(ST)^3 = \tau S^2$. Furthermore, this operator commutes with $T$, and $T_{0,0} = 1$.
\end{enumerate}
\end{definition}

\subsection{The universal grading of a modular category}

Let $(\cR,\bold{B})$ be a fusion ring and let $G$ be a finite group. A \emph{grading} of $\cR$ by $G$ is a map $|-| : B \to G$ that satisfies: for all $i,j,k \in \bold{B}$ with $N_{i,j}^k\neq 0$, we have $|k| = |i||j|$. We will associate a grading with its corresponding $G$-grading of the ring $\cR$ given by
\begin{equation*}
\cR= \bigoplus_{g\in G}\cR_g,
\end{equation*}
where $\cR_g$ is the $\mathbb{Z}$-submodule generated by all elements in $B$ of degree $g \in G$.
The $\mathbb{Z}$-submodule $\cR_1$ is termed the \emph{trivial component} of the grading. A grading is called faithful if the map $|-| :\bold{B} \to G$ is surjective.
For any fusion ring $\cR$, there exists a universal grading $\bold{B} \to U(\cR)$, where $U(\cR)$ denotes the universal grading group of $\cR$, see \cite{nilpotent}. Every grading of $\cR$ arises from a quotient of $U(\cR)$. The trivial component of the universal grading is the adjoint fusion subring $\cR_{ad} \subset \cR$, generated by all elements $b_k \in B$ for which $N_{i,i^*}^k\neq 0$ for some $b_i\in \bold{B}$.

Given a fusion ring $(\cR, \bold{B})$, we denote the \emph{group of units} of the fusion ring $\cR$ as $G(\cR) := \{ b_i\in \bold{B} : b_ib_{i^*}=1 \}$, using the product of $\cR$. In the case of $K_0(\cC)$, the Grothendieck ring of a fusion category $\cC$, the group $G(K_0(\cC))$ corresponds to the group of $\otimes$-invertible objects in $\cC$.

The following result has been previously established for modular tensor categories in \cite{nilpotent}. However, it can be generalized to any modular data, provided that a certain technical condition holds. The proof remains identical to that of \cite[Theorem 6.3]{nilpotent} and is therefore omitted here.

\begin{proposition}\cite[Theorem 6.2]{nilpotent} \label{grading modular data}
Let $M$ be a modular datum. Assume that for each unit $j \in I$, the equation
\begin{equation*}\label{eq:universal_grading_condition}
\frac{s_{kj}}{d_k d_j} = \frac{s_{ij}}{d_i d_j} \frac{s_{lj}}{d_l d_j}
\end{equation*}
is satisfied for all indices. Consider the mapping
\begin{align*}
|-| : \mathcal{R}_M &\to \widehat{G(\mathcal{R}_M)} \\
b_i &\mapsto [b_j\mapsto \frac{s_{ij}}{s_{0i}s_{0j}}], \notag
\end{align*}
where $\widehat{G(\mathcal{R}_M)}$ is the abelian group of linear characters of $G(\mathcal{R}_M)$. This mapping induces a faithful grading on $\mathcal{R}_M$. Moreover, $U(\mathcal{R}_M) \cong \widehat{G(\mathcal{R}_M)}$.
\end{proposition}
\qed

\section{Zesting of Modular Data}\label{section:zesting}

\subsection{Preliminaries on group cohomology}

In this section, we recall some basic definitions on group cohomology, abelian group cohomology, and the cup product in order to fix notation.

\subsubsection{Group cohomology}
To establish notation to the standard cocycle description of group cohomology, we present basic definitions in this section. For a detailed exposition, we refer the reader to \cite{KARPY}.

Given a group $G$ and an abelian group $B$, let $\big (C^n(G,B), d \big )$ be the cochain complex 
\[
0\to C^0(G,B)\to C^1(G,B)\to C^2(G,B)\to \cdots \to C^n(G,B)\to \cdots,
\]
with $C^0(G,B) = B$. Here, $C^n(G,B)$ represents the abelian group of all maps from $G^{\times n}$ to $B$, and the differential $d^n: C^n(G,B) \to C^{n+1}(G,B)$ is defined as
\begin{align}
d^n(f)(g_1,\ldots, g_{n+1}) &= f(g_2,\ldots,g_{n+1}) + \sum_{i=1}^{n} (-1)^i f(g_1,\ldots,g_ig_{i+1},\ldots,g_{n+1}) \label{cochains G} \\
&\quad +(-1)^{n+1} f(g_1,\ldots,g_{n}). \notag 
\end{align}

The \emph{group cohomology} of $G$ with coefficients in $B$ is the cohomology derived from the cochain complex \eqref{cochains G}. Specifically, for each integer $n$, we set 
\[
H^n(G,B) := \frac{\ker(d^n)}{\operatorname{Im}(d^{n-1})}.
\]
Here, $Z^n(G,B) := \ker(d^n)$ refers to the group of $n$-cocycles, and $B^n(G,B) := \operatorname{Im}(d^{n-1})$ designates the group of $n$-coboundaries.

\subsubsection{The cup product in group cohomology}
We introduce a specialized cup product in group cohomology, essential for defining zesting. 

Let $G$ be a group and $B, A$ be abelian groups with a bilinear map $c: B \times B \to A$. The wedge product, or preliminary cup product, is defined as:
\begin{align*}
\alpha \wedge_c \beta: C^i(G,B) \otimes_{\mathbb{Z}} C^j(G,B) &\to C^{i+j}(G,A) \\
\alpha \otimes \beta &\mapsto c(\alpha(g_1, \ldots, g_i), \beta(g_{i+1}, \ldots, g_{i+j})).
\end{align*}

The key identity for the wedge product is:
\begin{equation*}
d^{i+j}(\alpha \wedge_c \beta) = d^i(\alpha) \wedge_c \beta + (-1)^i \alpha \wedge_c d^j(\beta),
\end{equation*}
which ensures the well-definedness of the cup product:
\begin{align}\label{eq: cup product}
\cup_c: H^i(G,B) \otimes_{\mathbb{Z}} H^j(G,B) &\to H^{i+j}(G,A),
\end{align}
where $\overline{\alpha} \cup_c \overline{\beta} := \overline{\alpha \wedge_c \beta}$.

If $c$ is symmetric, we have:
\begin{equation*}
\overline{\alpha} \cup_c \overline{\beta} = (-1)^{ij} \overline{\beta} \cup_c \overline{\alpha}.
\end{equation*}

\subsubsection{Abelian group cohomology}

For two abelian groups $A$ and $B$, the abelian cohomology groups $H^n_{\text{ab}}(A, B)$ are defined as $H^n(K(A, 2), B)$, see \cite{MR0065162}. We will recall the standard cocycle description of $H^3_{\text{ab}}(A, B)$: the abelian group of abelian 3-cocycles is

\begin{align*}
    Z^3_{\text{ab}}(A, B) &= \Big \{ (\omega, t) \in C^3(A, B) \oplus C^2(A, B) : \\
    &\quad d^4(\omega) = 0, \\
    &\quad t(y, z) - t(xy, z) + t(x, z) + \omega(x, y, z) - \omega(x, z, y) + \omega(z, x, y) = 0, \\
    &\quad t(x, y) - t(x, yz) + t(x, z) - \omega(x, y, z) + \omega(y, x, z) - \omega(y, z, x) = 0 \Big\},
\end{align*}

with differential map $\partial: C^2_{\text{ab}}(A, B) \to Z^3_{\text{ab}}(A, B)$ defined as follows:
\begin{align*}
    \partial(\kappa)(x, y, z) &= \kappa(y, z) - \kappa(xy, z) + \kappa(x, yz) - \kappa(x, y), \\
    \partial(\kappa)(x, y) &= \kappa(y, x) - \kappa(x, y).
\end{align*}

The cohomology group $H^3_{\text{ab}}(A, B)$ is defined as the quotient $Z^3_{\text{ab}}(A, B) / \operatorname{Im}(\partial)$.

It has been shown in \cite{MR0065162} that there exists a group  isomorphisms:
\[
    H^3_{\mathrm{ab}}(A, B) \cong \mathrm{Quad}(A, B),
\]
given by $(\omega,t) \mapsto q$, where $q(x) = t(x,x)$ for all $x \in A$.

Here, $\mathrm{Quad}(A, B)$ denotes the group of quadratic maps, i.e., maps $q:A \to B$ such that $q(x) + q(y) - q(x + y)$ defines a bilinear map from $A \times A$ to $B$ and satisfies $q(x) = q(-x)$ for all elements $x$ in $A$.

Quadratic forms over cyclic groups can be readily described, and their associated abelian 3-cocycles are easily constructed. Specifically, 

\begin{equation}\label{eq: abelian 3 cohomology of cyclic groups}
H^3_{\text{ab}}(\mathbb Z/N,B)\cong  \mathrm{Quad}(\mathbb{Z}/N, B) \overset{(*)}{\cong} \{ \nu \in B : N^2\nu = 2N\nu = 0 \},
\end{equation}
and the corresponding abelian 3-cocycle is given by:
\[
t_\nu(x, y) = xy\nu, \quad \omega_\nu(x, y, z) = 
\begin{cases} 
xN\nu, & \text{if } y + z \geq N, \\
0, & \text{otherwise.}
\end{cases}
\]
The bijections in \eqref{eq: abelian 3 cohomology of cyclic groups} are group isomorphisms, and the second isomorphism  $(*): \mathrm{Quad}(\mathbb{Z}/N, B) \to \{ \nu \in B : N^2 \nu = 2N \nu = 0 \}$ is simply $\lambda \mapsto \lambda(1)$.

For an arbitrary finite abelian group, we can use \cite[Proposition 4.3.]{Galindo-Jaramillo}, which states that for abelian groups $A_1$ and $A_2$, the map
\begin{align*}
    T : \text{Bil}(A_1, A_2; B) &\oplus \text{Quad}(A_1, B) \oplus \text{Quad}(A_2, B) \to \text{Quad}(A_1 \oplus A_2, B) \\
    f \oplus \gamma_{A_1} \oplus \gamma_{A_2} &\mapsto \left[ (a, b) \mapsto f(a , b) + \gamma_{A_1}(a) + \gamma_{A_2}(b) \right],
\end{align*}
is a group isomorphism. Here $\text{Bil}(A_1, A_2; B)$ is the abelian group of all bilinear maps from $A_1\times A_2$ to $B$.

Using the previous results, it is easy both to calculate the third abelian cohomology group and to find explicit representatives. 
The following result, although just an example, will be important for the parametrization of braided zesting of certain Verlinde categories.

\begin{proposition}\label{prop:abelian3CocycleOfKleinGroup}
We have
\begin{align*}
    H^3_{\text{ab}}(\mathbb{Z}/2 \times \mathbb{Z}/2, \mathbb{C}^*) = \mathbb{Z}/4 \oplus \mathbb{Z}/4 \oplus \mathbb{Z}/2.
\end{align*}
Representative abelian 3-cocycles are given by
\begin{equation}\label{eq: abelian 3-cocycles of klein}
\begin{aligned}
    \omega_{a,b,c}(\vec{x},\vec{y},\vec{z}) &= (-1)^{ax_1y_1z_1+bx_2y_2z_2}, \\
    t_{a,b,c}(\vec{x},\vec{y}) &= \zeta_4^{ax_1y_1+bx_2y_2}(-1)^{cx_1y_2}
\end{aligned}
\end{equation}
where $(a,b,c) \in \mathbb{Z}/4 \oplus \mathbb{Z}/4 \oplus \mathbb{Z}/2$ and $\zeta_4$ is a fourth root of unity.
\end{proposition}

\qed

\subsection{Associative zesting data}

Consider a modular category $\mathcal{B}$. In this section, we want to adapt Definition 3.2 in \cite{zesting} of associative $U(\mathcal{B})$-zesting to modular data. For this purpose, both $\mathcal{B}_{\text{pt}}$, the category of invertible objects in $\mathcal{B}$, and the subcategory of invertible objects in the trivial component of the $U(\mathcal{B})$-grading, that is, in $\mathcal{B}_{\text{ad}}$, are the only ones that play a role in the process. We will detail these categories carefully before presenting the main definition.

It was shown in \cite[Corollary 8.22.8]{EGNO} that $\mathcal{B}_{\text{ad}} = \mathcal{B}_{\text{pt}}'$, where $\mathcal{B}_{\text{ad}}$ is the adjoint fusion subcategory and $\mathcal{B}_{\text{pt}}'$ is the full 
subcategory of objects $X$ in $\mathcal{B}$ such that $c_{Y,X} \circ c_{X,Y} = \id_{X \otimes Y}$ 
for all $Y \in \mathcal{B}_{\text{pt}}$. 
This implies that $\mathcal{B}_{\text{ad}} \cap \mathcal{B}_{\text{pt}}$, the maximal pointed subcategory of $\mathcal{B}_{\text{ad}}$, is a pointed symmetric fusion category and corresponds to 
\[
\mathcal{Z}_2(\mathcal{B}_{\text{pt}}) = \{X \in \mathcal{B}_{\text{pt}} \mid c_{Y,X} \circ c_{X,Y} = \id_{X \otimes Y} \text{ for all } Y \in \mathcal{B}_{\text{pt}} \}.
\]

Every pointed symmetric fusion category, up to equivalence, has the form $\text{Vec}_{A_0}^\nu$, where $A_0$ is an abelian group and $\nu: A_0 \to \{\pm 1\}\cong \mathbb{Z}/2$ is a group homomorphism that corresponds to the ribbon structure. The braiding is given by the symmetric bicharacter
\begin{equation}\label{eq: definition bicharacter symmetric}
c_\nu(a, b) = (-1)^{\nu(a) \nu(b)} \id_{a \otimes b}, \quad \quad \nu(a), \nu(b)\in \{0,1\}.
\end{equation}
In particular, $\text{Vec}_{A_0}^\nu$ is Tannakian if $\nu$ is trivial and super-Tannakian if $\nu$ is non-trivial.

Consider a modular datum $M = (I, S, T)$. The abelian group $G(\mathcal{R}_M)$ of units of $\mathcal{R}_M$ with the quadratic form 
\begin{align*}
    q_T: G(\mathcal{R}_M) \to \mathbb{C}^* \\
    b_i \mapsto T_{ii}
\end{align*}
corresponds to the pointed braided category $\mathcal{B}_{\text{pt}}$. Then 
\[G(\mathcal{R}_M)_0 = \{b_i \in G(\mathcal{R}_M) \mid q_T(b_i b_j) = q_T(b_i) q_T(b_j) \text{ for all } b_j \in G(\mathcal{R}_M)\}\]
corresponds to the pointed symmetric fusion category $\mathcal{B}_{\text{ad}} \cap \mathcal{B}_{\text{pt}}$. Here, the quadratic form $q_T$ restricted to $G(\mathcal{R}_M)_0$ is a group homomorphism to $\{\pm 1\}$ and it has a canonical symmetric bicharacter 
\[c_T: G(\mathcal{R}_M)_0 \times G(\mathcal{R}_M)_0 \to \{\pm 1\}\]
defined as in equation \eqref{eq: definition bicharacter symmetric} such that $c_T(b_i,b_i)=T_q(b_i)=T_{ii}$ for all $b_i\in G(\mathcal{R}_M)_0$.

\begin{definition}\label{def: premetirc group A0 y c}
\begin{itemize}
    \item A \emph{pre-metric group} is a pair $(A, q)$ where $A$ is a finite abelian group and $q: A \to \mathbb{C}^*$ is a map such that the symmetric function $b(x, y) := \frac{q(x) q(y)}{q(xy)}$ is a bicharacter and $q(x) = q(x^{-1})$ for all $x \in A$.
    \item The kernel of the bicharacter $b$ is denoted by $A_0$, that is 
    \[A_0 := \{x \in A \mid q(xy) = q(x) q(y) \text{ for all } y \in A\}.\]
    \item A symmetric bicharacter on $A_0$ is defined by
    \begin{equation}\label{eq: bicharacter symmetric}
    c_q(x, y) = (-1)^{\widehat{q}(x) \widehat{q}(y)}, \quad x, y \in A_0,
    \end{equation}
    where $\widehat{q}(x) = 1$ if $q(x) = -1$ and $\widehat{q}(x) = 0$ if $q(x) = 1$.
\end{itemize}
In general, to simplify notation, we will write $c$ instead of $c_q$, though it is canonically defined by $(A,q)$.
\end{definition}

The following definition is adapted from \cite[Definitions 3.2, 4.1]{zesting}.

\begin{definition}
Let $(A, q)$ be a pre-metric group. An \emph{associative zesting datum} of $(A, q)$ consists of a pair $(\lambda, \omega)$. Here, $\lambda \in Z^2(\widehat{A}, A_0)$ is a 2-cocycle and $\omega \in C^3(\widehat{A}, \mathbb{C}^*)$ is a 3-cochain, satisfying:
\begin{equation*}\label{eq:BZ1}
    d^3(\omega) = \lambda \wedge_c \lambda,
\end{equation*}
where $\lambda \wedge_c \lambda(\chi_1, \chi_2, \chi_3, \chi_4) = c(\lambda(\chi_1, \chi_2), \lambda(\chi_3, \chi_4))$. Here, the subgroup $A_0$ and the bicharacter $c$ are as in Definition \ref{def: premetirc group A0 y c}.
\end{definition}

\begin{remark}
Given a modular category $\mathcal{B}$, we have an associated pre-metric group corresponding to the pointed braided fusion category $\mathcal{B}_{\text{pt}}$. The associative zestings of the pre-metric group of $\mathcal{B}_{\text{pt}}$ are in non-canonical bijective correspondence with associative $U(\mathcal{B})$-zestings of $\mathcal{B}$, as defined in \cite[Definition 3.2]{zesting}. Here, using the notation of \cite{zesting}, we have $\mathcal{B}_e \subset R_{\mathcal{B}_e}(\mathcal{B})$ via the braiding.
\end{remark}

Observe that if $(\lambda, \omega)$ is an associative zesting datum and $\lambda' = \lambda d^1(\psi)$ for some $\psi \in C^1(\widehat{A}, A_0)$, then the following holds:
\[
\lambda' \wedge_c \lambda' = \lambda \wedge_c \lambda \times d^3\left( \lambda \wedge_c \psi \times \psi \wedge_c \lambda' \right).
\]
Consequently, the pair $(\lambda', \omega \times \lambda \wedge_c \psi \times \psi \wedge_c \lambda')$ also forms an associative zesting datum. Motivated by this observation, we introduce the following definition of equivalence between associative zesting data.

\begin{definition}
Let $(A, q)$ be a pre-metric group. Two associative zesting data $(\lambda, \omega)$ and $(\lambda', \omega')$ are said to be equivalent if there exists a pair $(\psi, \kappa)$, where $\psi \in C^1(\widehat{A}, A_0)$ and $\kappa \in C^2(\widehat{A}, \mathbb{C}^*)$, such that $d^1(\psi) = \frac{\lambda'}{\lambda}$ and $d^2(\kappa) = \frac{\omega'}{\omega \times \lambda \wedge_c \psi \times \psi \wedge_c \lambda'}$.

We will denote by $\mathbf{Sez}(A, q)$ the set of equivalence classes of associative zesting data of $(A, q)$.
\end{definition}

 A first observation that will be important for the classification and computation of all zesting data is that the abelian group $H^3(\widehat{A}, \mathbb{C}^*)$ acts freely on $\mathbf{Sez}(A,q)$. In fact, if $\tau \in Z^3(\widehat{A}, \mathbb{C}^*)$ and $(\lambda, \omega)$ is an associative zesting datum, then
\[
\tau \cdot (\lambda, \omega) = (\lambda, \omega  \tau)
\]
defines an action of $Z^3(\widehat{A},\mathbb C^*)$ on the set of associative zesting data that descends to a free action of $H^3(\widehat{A},\mathbb C^*)$  on equivalence classes.

The proof of the following proposition follows directly from the definition; however, it is important for the construction and classification of zesting up to equivalence.

\begin{proposition}
Let $(A,q)$ be a pre-metric group. We have an exact sequence of pointed sets
\begin{equation*}\label{eq: exact sequence associative zesting}
0 \to H^3(\widehat{A}, \mathbb{C}^*) \to \mathbf{Sez}(A,q) \stackrel{pr_1}{\to} H^2(\widehat{A}, A_0) \stackrel{\cup_c }{\to} H^4(\widehat{A}, \mathbb{C}^*),
\end{equation*}
inducing a short exact sequence
\begin{equation}\label{eq: sequence associ zesting}
0 \to H^3(\widehat{A}, \mathbb{C}^*) \to \mathbf{Sez}(A,q) \to \operatorname{ker}(\cup_c) \to 0,
\end{equation}
where \( H^3(\widehat{A}, \mathbb{C}^*) \), \( H^2(\widehat{A}, A_0) \), and \( H^4(\widehat{A}, \mathbb{C}^*) \) are pointed by their identities. The set \( \mathbf{Sez}(A,q) \) is pointed by the class of the trivial zesting datum. The map \( \cup_c \) is defined in \eqref{eq: cup product}, and
\[
\operatorname{ker}(\cup_c) = \{ \lambda \in H^2(\widehat{A}, A_0) : \lambda \cup_c \lambda = 0 \}.
\]
Thus, the set of \( H^3(\widehat{A}, \mathbb{C}^*) \)-orbits in \( \mathbf{Sez}(A,q) \), denoted \( \mathbf{Sez}(A,q)/H^3(\widehat{A}, \mathbb{C}^*) \), can be identified with \( \operatorname{ker}(\cup_c) \).
\end{proposition}

\qed

\begin{example}
Assume that $\{\unit, f\}=A_0 = A = \mathbb{Z}/2$ with $\theta_f = -1$. In this case, $H^3(\mathbb{Z}/2, \mathbb{C}^*) = \mathbb{Z}/2$ and $\ker(\cup_c) = H^2(\mathbb{Z}/2, \mathbb{Z}/2) \cong \mathbb{Z}/2$. An associative zesting datum with a non-trivial 2-cocycle is given by  $\lambda(f, f) = 1$ and $\omega(f, f, f) = i$, where we are identifying $\widehat{A}$ and $A$.
\end{example}

\subsection{Braided zesting data}

\begin{definition}
Let $(A,q)$ be a pre-metric group. A \emph{braided zesting datum} of $(A,q)$ is a triple $(\lambda, \omega, t)$ where $(\lambda,\omega)$ is an associative zesting datum and $t \in C^2(\widehat{A}, \mathbb{C}^*)$ is 2-cochain satisfying the following conditions:

\begin{enumerate}
\item The 2-cocycle $\lambda \in Z^2(\widehat{A}, A_0)$ is symmetric, i.e.,
\begin{align*}
    \lambda(\chi_1, \chi_2) = \lambda(\chi_2, \chi_1),
\end{align*}
for all $\chi_1, \chi_2 \in \widehat{A}$.

\item 

\begin{equation}\label{eq:BZ2}
    \frac{\omega(\chi_1, \chi_2, \chi_3) \omega(\chi_2, \chi_3, \chi_1)}{\omega(\chi_2, \chi_1, \chi_3)} = \frac{t(\chi_1, \chi_2) t(\chi_1, \chi_3)}{t(\chi_1, \chi_2 \chi_3)},
\end{equation}

\begin{equation}\label{eq:BZ3}
    \chi_3(\lambda(\chi_1, \chi_2)) = \frac{\omega(\chi_1, \chi_2, \chi_3) \omega(\chi_3, \chi_1, \chi_2)}{\omega(\chi_1, \chi_3, \chi_2)} \frac{t(\chi_1, \chi_3) t(\chi_2, \chi_3)}{t(\chi_1 \chi_2, \chi_3)},
\end{equation}
for all $\chi_1, \chi_2, \chi_3 \in \widehat{A}$.
\end{enumerate}
\end{definition}

Observe that if $(\lambda, \omega, t)$ is a braided zesting datum of a pre-metric group $(A, q)$, and $\psi \in C^1(A, A_0)$, then $\lambda' = \lambda d^1(\psi)$ is also a symmetric 2-cocycle. Moreover, $\omega' = \omega \times \lambda \wedge_c \psi \times \psi \wedge_c \lambda'$ and $t' = t \times \psi \wedge_c \psi$ together form another braided zesting datum. Based on this observation, we introduce the following equivalence.

\begin{definition}\label{def: equivalent braided zestings}
Two braided zesting data $(\lambda, \omega,t)$ and $(\lambda', \omega',t')$ are said to be equivalent if there exists a pair $(\psi, \kappa)$ where $\psi \in C^1(\widehat{A}, A_0)$ and $\kappa \in C^2(\widehat{A}, \mathbb{C}^*)$ such that $\lambda' = \lambda \times d^1(\psi)$ and
\begin{align*}
    \frac{\omega'}{\omega \times \lambda \wedge_c \psi \times \psi \wedge_c \lambda'}(\chi_1, \chi_2, \chi_3) &= \frac{\kappa(\chi_1 \chi_2, \chi_3) \kappa(\chi_1, \chi_2)}{\kappa(\chi_1, \chi_2 \chi_3) \kappa(\chi_2, \chi_3)}, \\
    \frac{t'}{t \times \psi \wedge_c \psi}(\chi_1, \chi_2) &= \frac{\kappa(\chi_1, \chi_2)}{\kappa(\chi_2, \chi_1)}
\end{align*}
for all $\chi_1, \chi_2, \chi_3 \in \widehat{A}$.

We will denote by $\mathbf{Sez}_{\text{br}}(A,q)$ the set of equivalence classes of braided zesting data of $(A,q)$.
\end{definition}

\begin{remark}\label{Remark braided and abelian cocyles}
  The notion of braided zestings for a pre-metric group $(A,q)$ is closely related to the concept of abelian 3-cocycles. In fact, a braided zesting datum $(\lambda, \omega, t)$, where $\lambda$ is the trivial 2-cocycle, corresponds precisely to an \emph{abelian 3-cocycle} of $A$.  
\end{remark}

Similarly to the associative case, the abelian group $H^3_{\text{ab}}(\widehat{A}, \mathbb{C}^*)$ acts freely on $\mathbf{Sez}_{\text{br}}(A, q)$. Hence, there is a pointed set exact sequence

\begin{align}\label{short sequence}
0 \to H^3_{\text{ab}}(\widehat{A}, \mathbb{C}^*) \to \mathbf{Sez}_{\text{br}}(A, q) &\to H^2_{\operatorname{ab}}(\widehat{A}, A_0),\\
[(\lambda, \omega, t)] &\mapsto [\lambda] \notag
\end{align}
where $H^3_{\text{ab}}(\widehat{A}, \mathbb{C}^*)$ is the third abelian group cohomology group, and $H^2_{\operatorname{ab}}(\widehat{A}, A_0)$ is the second abelian group cohomology of symmetric 2-cocycles, or equivalently, $\operatorname{Ext}(\widehat{A}, A_0)$, the group of abelian extensions of $A$ by $A_0$.

Following \cite[Section 8.6]{davydov2021braided}, and in order to formulate a long exact sequence that extends \eqref{short sequence}, we introduce for abelian groups $A$ and $B$ the cohomology group $\widetilde{H}^4_{\text{ab}}(\widehat{A}, B) := \widetilde{Z}^4_{\text{ab}}(A, B) / \operatorname{Im}(\widetilde{\partial})$ (a subgroup of $H^4_{\text{ab}}(A, B)$) defined as follows:

\begin{align*}
\widetilde{Z}^4_{\text{ab}}(A, B) = \Big \{ (\alpha, \gamma) \in C^4(A, B) \oplus C^3(A, B) : 
\quad\quad\quad\quad\quad\quad\quad d^4(\alpha) &= 0, \\
\quad - \alpha(x, y, z, w) + \alpha(y, x, z, w) - \alpha(y, z, x, w) + \alpha(y, z, w, x) &= 0, \\
\gamma(x \mid z, w) - \alpha(x \mid yz, w) + \gamma(x \mid y, zw) - \gamma(x \mid y, z) \\
\alpha(x, y, w, z) - \alpha(x, y, z, w) - \alpha(w, y, z, x) &= 0, \\
\alpha(y \mid z, w) - \gamma(xy \mid z, w) + \gamma(x \mid yz, w) - \gamma(x \mid y, zw) \\
\quad + \alpha(x, y, z, w) - \alpha(z, x, w, y) + \alpha(z, w, x, y) &= 0, \\
\gamma(x \mid y, z) - \gamma(x \mid z, y) &= 0 \Big\},
\end{align*}

\begin{align*}
\widetilde{C}_{\text{ab}}^3(A, B) = \Big\{ (\omega, t) \in C^3(A, B) \oplus C^2(A, B) : & \\
t(y, z) - t(xy, z) + t(x, z) &= \omega(x, z, y) + \omega(x, y, z) - \omega(z, x, y) \Big\}.
\end{align*}
with differential map $\widetilde{\partial}: \widetilde{C}_{\text{ab}}^3(A, B) \to \widetilde{Z}^4_{\text{ab}}(A, B)$ defined as follows:
\begin{align*}
   \widetilde{\partial}(\omega) &= d^3(\omega), \\
   \widetilde{\partial}(\kappa)(x, y) &= t(x, y) - t(x, yz) + t(x, z) - \omega(x, y, z) + \omega(y, x, z) - \omega(y, z, x).
\end{align*}

Following the definition of the Pontryagin–Whitehead quadratic function given in \cite[Equation (8.44)]{davydov2021braided}, a group homomorphism
\[
PW^2: H^2_{\text{ab}}(\widehat{A}, A_0) \to \widetilde{H}^4_{\text{ab}}(A, \mathbb{C}^*)
\]
is defined as follows:

\begin{align*}
PW^2(\lambda) = \lambda \wedge_c \lambda,  &&
PW^2(\lambda)(\chi_1 , \chi_2, \chi_3) = \langle \chi_1, \lambda(\chi_2, \chi_3) \rangle,
\end{align*}
where $\langle \chi, a \rangle$ denotes the evaluation of $a \in A_0$ at $\chi \in \widehat{A}$.

Note that by the definition of $\widetilde{H}^4_{\text{ab}}(\widehat{A}, \mathbb{C}^*)$, we have $\lambda \in \operatorname{ker}(PW^2)\subset H^2_{\text{ab}}(\widehat{A}, A_0)$ if and only if $\lambda$ admits a braided zesting. This leads to the following result:

\begin{proposition}\label{prop: characterization of image}
Let $(A,q)$ be a pre-metric group. Then the pointed set of equivalence classes of braided zesting data fits into the following exact sequence:
\[
0 \to H^3_{\text{ab}}(\widehat{A}, \mathbb{C}^*) \to \mathbf{Sez}_{\text{br}}(A,q) \stackrel{pr_1}{\to} H^2_{\text{ab}}(\widehat{A}, A_0) \stackrel{PW^2}{\to} \widetilde{H}^4_{\text{ab}}(\widehat{A}, A_0).
\]
In particular, we obtain the short exact sequence 
\[
0 \to H^3_{\text{ab}}(\widehat{A}, \mathbb{C}^*) \to \mathbf{Sez}_{\text{br}}(A,q) \to \operatorname{ker}(PW^2) \to 0.
\]
\end{proposition}
\qed

\begin{example}\label{Ex: fermion zesting}
Consider $A = A_0 = \mathbb{Z}/2 = \langle f \rangle$ and $q(f) = -1$, which implies that $(A,q)$ is the pre-metric group associated to the symmetric category of super vector spaces. We have that $H^3_{\text{ab}}(\mathbb{Z}/2, \mathbb{C}^*) \cong \mathbb{Z}/4$, with a generator given by the abelian 3-cocycle $\tau(f, f, f) = -1$ and the 2-cochain $c(f, f) = i$. Additionally, $H^2_{\text{ab}}(\mathbb{Z}/2, \mathbb{Z}/2) = H^2(\mathbb{Z}/2, \mathbb{Z}/2) = \mathbb{Z}/2$.

A braided zesting for the non-trivial 2-cocycle $\lambda(f, f) = f$, has braided zesting datum $\omega(f, f, f) = i$, and $t(f, f) = e^{\frac{\pi i}{4}}$. For this example, we employ the canonical identification between $\widehat{\mathbb{Z}/2}$ and $\mathbb{Z}/2$.
\end{example}

\subsection{Sequential Obstructions for Braided Zestings}\label{subsecton: sequencial obstructions}

Although Proposition \ref{prop: characterization of image} offers a  characterization of the image of the map $\mathbf{Sez}_{\text{br}}(A,q) \stackrel{pr_1}{\to} H^2_{\text{ab}}(\widehat{A}, \mathbb{C}^*)$, it remains important to develop an algorithmic approach for determining the triviality of $PW^2(\lambda)$ in computational settings. The first obstruction, naturally, is that $\lambda \cup_c \lambda = 0$. The second obstruction is expressed using a straightforward identity for the 3-cochain, ensuring $d^3(\omega)=\lambda \wedge_c \lambda$.

\begin{proposition}\label{prop:Obstruction1}
Let $(A, q)$ be a pre-metric group and $(\lambda, \omega)$ an associative zesting datum. The cochain $\omega$ admits a $t \in C^2(\widehat{A}, \mathbb{C}^*)$ satisfying equation \eqref{eq:BZ2} if and only if
\begin{equation}\label{eq:AlternatingProduct}
\prod_{\sigma \in \mathbb{S}_3} \omega\left(\chi_{\sigma(1)}, \chi_{\sigma(2)}, \chi_{\sigma(3)}\right)^{\operatorname{sgn}(\sigma)} = 1,
\end{equation}
holds for all $\chi_1, \chi_2, \chi_3 \in \widehat{A}$. Here, $\operatorname{sgn}(\sigma)$ denotes the sign of the permutation $\sigma\in \mathbb{S}_3$.
\end{proposition}

\begin{proof}
We begin by noting that the existence of a 2-cochain $t \in C^2(\widehat{A}, \mathbb{C}^*)$ satisfying equation \eqref{eq:BZ2} is characterized by \cite[Lemma 4.12]{zesting}. Specifically, such a $t$ exists if and only if, for each $\chi_1 \in \widehat{A}$, the 2-cocycle $O_1(\lambda)(\chi_1| -, -) \in Z^2(\widehat{A}, \mathbb{C}^*)$ defined by
\[
O_1(\lambda)(\chi_1 | \chi_2, \chi_3) := \frac{\omega(\chi_1, \chi_2, \chi_3) \omega(\chi_2, \chi_3, \chi_1)}{\omega(\chi_2, \chi_1, \chi_3)}
\]
has a trivial cohomology class. Since $\mathbb{C}^*$ is a divisible group, a 2-cocycle over the abelian group $\widehat{A}$ has trivial cohomology if and only if it is symmetric.

We now observe that
\[
\frac{O_1(\lambda)(\chi_1| \chi_2, \chi_3)}{O_1(\lambda)(\chi_1| \chi_3, \chi_2)}=\prod_{\sigma \in \mathbb{S}_3} \omega\left(\chi_{\sigma(1)}, \chi_{\sigma(2)}, \chi_{\sigma(3)}\right)^{\operatorname{sgn}(\sigma)} ,
\]
which implies that $O_1(\lambda)(\chi_1 | \cdot, \cdot)$ is symmetric for all $\chi_1$ if and only if condition \eqref{eq:AlternatingProduct} is satisfied.
\end{proof}

\begin{remark}\label{rmk:Psi}
For an arbitrary pair of groups $A$ and $B$, the map $\Psi: C^3(A,B) \to C^3(A,B)$, given by 
\[
\Psi(\omega )(x_1,x_2,x_3)=\sum_{\sigma \in \mathbb{S}_3} \operatorname{sgn}(\sigma)\omega\left(x_{\sigma(1)}, x_{\sigma(2)}, x_{\sigma(3)}\right),
\]
is a group homomorphism. When restricted to $Z^3(A,B)$, it maps to $\Hom(\wedge^3 A, B)$, the group of alternating trilinear maps. Furthermore, $\Psi(\tau) = 0$ for any $\tau \in B^3(A,B)$, thereby inducing a group homomorphism from $H^3(A,B)$ to $\Hom(\wedge^3 A, B)$. 
\end{remark}

Finally, the following proposition provides the final step for verifying the triviality of $PW^2(\lambda)$.

\begin{proposition}\cite[Corollary 4.16]{zesting}\label{Prop: obstruccion O2}
Let $(\lambda, \omega, l)$ be such that $(\lambda, \omega)$ is an associative zesting datum and $l \in C^2(\widehat{A}, \mathbb{C}^*)$ satisfies equation \eqref{eq:BZ2}.

Define the map $O_2(- | \chi_2, \chi_3) \in C^1(\widehat{A}, \mathbb{C}^*)$ as
\[
O_2 (\chi_1 | \chi_2, \chi_3) = \frac{\omega(\chi_1, \chi_2, \chi_3) \omega(\chi_3, \chi_1, \chi_2)}{\omega(\chi_1, \chi_3, \chi_2)} \frac{l(\chi_1, \chi_3) l(\chi_2, \chi_3)}{l(\chi_1 \chi_2, \chi_3)},
\]
where $\chi_2, \chi_3 \in \widehat{A}$.

This map is a linear character that defines a 2-cocycle $O_2(\omega,l) \in H^2(\widehat{A}, A)$. Furthermore, an element $t \in C^2(\widehat{A}, \mathbb{C}^*)$ exists that satisfies equations \eqref{eq:BZ2} and \eqref{eq:BZ3}, implying that $PW^2(\lambda) = 0$, if and only if
\[
[O_2(\omega,l)] = [\lambda] \in H^2(\widehat{A}, A).
\]
\end{proposition}
\qed

\begin{remark}\label{rmk: procedimiento para calcular braided zestings}
We can outline an algorithmic procedure for calculating the braided zesting data of a pre-metric group $(A,q)$  where $q(x) = -1$ for some $x\in A_0$. The steps are as follows:

\begin{enumerate}
    \item   Compute symmetric 2-cocycles that represent the classes in $\lambda \in H^2_{\text{ab}}(\widehat{A}, A_0)$ such that $\lambda \cup_c \lambda = 0$.
    \item Given an associative zesting datum $(\lambda, \omega)$, check if $\Psi(\omega) = \Psi(\omega')$ for some $\omega' \in Z^3(\widehat{A}, \mathbb{C}^*)$ (see Remark \ref{rmk:Psi}). In this case, $(\lambda, \omega/\omega')$ is an associative zesting datum satisfying equation \eqref{eq:AlternatingProduct}. If not, this $\lambda$ does not admit a braided zesting.
    
    \item If equation \eqref{eq:AlternatingProduct} holds, compute $l \in C^2(\widehat{A}, \mathbb{C}^*)$ that satisfies equation \eqref{eq:BZ2}.
    \item Following Proposition \eqref{Prop: obstruccion O2}, verify that $[O_2(\omega,l)] = [\lambda]$ in $H^2_{\text{ab}}(\widehat{A}, A_0)$. If $[O_2(\omega)] \neq [\lambda]$, then $\lambda$ does not admit a braided zesting.
\end{enumerate}

By following these steps, we can calculate representatives of the $H^3_{\text{ab}}(\widehat{A}, \mathbb{C}^*)$-orbits of the braided zesting data.

\end{remark}
 
\subsection{Ribbon zesting data}

Considering the results of \cite[Corollary 5.3]{zesting} on the definition of ribbon zesting data for universally graded braided fusion categories, we introduce the corresponding notion for pre-metric groups.

\begin{definition}
Let $(A, q)$ be a pre-metric group. A \emph{ribbon zesting datum} of $(A, q)$ is a quadruple $(\lambda, \omega, t, f)$ where $(\lambda, \omega, t)$ is a braided zesting datum and $f: \widehat{A} \to \mathbb{C}^*$ is a 1-cochain satisfying the following conditions:
\begin{align}
    f(\chi_1 \chi_2) \langle \chi_1 \chi_2, \lambda(\chi_1, \chi_2) \rangle q(\lambda(\chi_1, \chi_2)) &= f(\chi_1) f(\chi_2) \frac{t(\chi_1, \chi_2)}{t(\chi_2, \chi_1)}, \label{ribbon eq 1} \\
    \frac{f(\chi_1)}{f(\chi_1^{-1})} &= q(\lambda(\chi_1, \chi_1^{-1})) \langle \chi_1, \lambda(\chi_1, \chi_1^{-1}) \rangle \label{ribbon eq 2}
\end{align}
for all $\chi_1, \chi_2 \in \widehat{A}$.
\end{definition}

Note that if $(\lambda, \omega, t, f)$ and $(\lambda, \omega, t, f')$ are ribbon zesting data, then $\frac{f}{f'}: \widehat{A} \to \{\pm 1\}$ is a group homomorphism. Hence, the set of all possible ribbon zesting data structures associated with a fixed braided zesting datum forms a torsor over the abelian group $\operatorname{Hom}(\widehat{A}, \{\pm 1\}) \cong A_2$, where $A_2$ is the 2-torsion subgroup of $A$.

\subsection{Zested  modular data}

A fusion category is called \emph{pseudo-unitary} if its Frobenius-Perron dimension coincides with the global dimension, as discussed in \cite{ENO}. It was established in \cite[Proposition 8.23]{ENO} that a \emph{pseudo-unitary} fusion category has a unique spherical structure, such that the quantum dimension of each object is positive (and thus coincide with the Frobenius-Perron dimension). 

We will extend this terminology to modular categories, referring to them as \emph{pseudo-unitary} if the quantum dimensions of their objects are positive. Accordingly, we define a modular datum $M$ to be \emph{pseudo-unitary} if $d_i \in \mathbb{R}^{>0}$ for every $i \in I$.

In this section, we introduce a formula for the new pseudo-unitary modular data derived from universally graded ribbon zestings of a pseudo-unitary modular category. Although these formulas may be less general than those presented in \cite{zesting}, they offer the advantages of simplicity and symmetry. Moreover, modular data corresponding to Verlinde categories are pseudo-unitary, and unitary modular categories hold particular importance in various applications, such as topological quantum computation and the study of topological phases of matter \cite{RW}.

The advantage of considering a pseudo-unitary modular category $\mathcal{B}$ is that for a braided zesting datum $(\lambda, \omega, t)$, there is a unique choice for the ribbon zesting function that produces a positive quantum dimension on $\mathcal{B}^\lambda$, given by $f(\chi) = t(\chi, \chi)$. This is established in the following theorem.

\begin{theorem} \label{th:new_modular_data}
Let $\mathcal{B}$ be a pseudo-unitary modular category. A ribbon zesting datum $(\lambda, \omega, t, f)$ of the pre-metric group $(G(\mathcal{B}),\theta)$ associated with $\mathcal{B}_{\text{pt}}$ produces a pseudo-unitary zesting $\mathcal{B}^\lambda$ if and only if 
\[
f(\chi) = t(\chi, \chi), \quad \forall \chi \in \widehat{G(\mathcal{B})}.
\]
The corresponding modular data is given by
\begin{equation}\label{eq:formula zested modular data}
s^{\lambda}_{i,j} = t(|i|, |j|) t(|j|, |i|) s_{i,j}, \quad T^{\lambda}_{i,i} = \frac{T_{i,i}}{t(|i|, |i|)},
\end{equation}
where $i, j \in \text{Irr}(\mathcal{B})$ and $|i|, |j| \in \widehat{G(\mathcal{B})}$ correspond to the grading elements.
\end{theorem}

\begin{proof}
Recall that the fusion rules for the zested categories are given by $X_{\chi_1} \otimes^{\lambda} Y_{\chi_2} = \lambda(\chi_1, \chi_2) \otimes X_{\chi_1}\otimes Y_{\chi_2}$. In particular, the Frobenius-Perron dimension function $\text{FP-dim}: K_0(\mathcal{B}^\lambda) \to \mathbb{R}$ defines a ring homomorphism $\text{FP-dim}: K_0(\mathcal{B}^\lambda) \to \mathbb{R}$. Indeed, we have
\[
\text{FP-dim}(\lambda(\chi_1, \chi_2)\otimes X_{\chi_1}\otimes Y_{\chi_2}) = \text{FP-dim}(X_{\chi_1}) \text{FP-dim}(Y_{\chi_2}),
\]
which implies that the Frobenius-Perron dimension remains invariant under zesting.

Consequently, if both $\mathcal{B}$ and $\mathcal{B}^\lambda$ are pseudo-unitary, then $\dim_{\mathcal{B}}(X_\chi) = \text{FP-dim}(X_\chi) = \dim_{\mathcal{B}^\lambda}(X_\chi)$. However, by \cite[Proposition 5.5]{zesting}, we have
\[
\dim_{\mathcal{B}^\lambda}(X_\chi) = \frac{t(\chi, \chi)}{f(\chi) \dim_{\mathcal{B}}(\lambda(\chi, \chi^{-1}))} \dim_{\mathcal{B}}(X_\chi),
\]
and using the fact that $\dim_{\mathcal{B}^\lambda}(\lambda(\chi, \chi^{-1})) = 1$, it follows that $\dim_{\mathcal{B}^\lambda}(X_\chi) = \text{FP-dim}(X_\chi)$ and hence pseudo-unitary if and only if $f(\chi) = t(\chi, \chi)$.

Finally, it follows from \cite[Lemma 5.6]{zesting} and the fact that $T_{i,i} = \theta_i^{-1}$ that equation \ref{eq:formula zested modular data} holds.
\end{proof}





\section{Computing some families of braided zestings}\label{section:families}

We adopt the same notation as in Definition \ref{def: premetirc group A0 y c}. Specifically, given a pre-metric group $(A, q)$, we denote by $A_0$ the kernel of the associated bicharacter $b(x, y) := \frac{q(x) q(y)}{q(xy)}$ and by $c: A_0 \times A_0 \to \{\pm 1\}$ the symmetric bicharacter given by formula \eqref{eq: bicharacter symmetric}.
If $q(x) = 1$ for all $x \in A_0$, we refer to this as a \emph{Tannakian case}; otherwise, it is a \emph{super-Tannakian case}. The objective of this section is to elucidate how to compute braided zestings in both the Tannakian and super-Tannakian cases, with particular emphasis on cases that are especially relevant for Verlinde modular categories.

\subsection{Tannakian Zesting}
As we explain in Remark \ref{rmk:Psi}, for a finite abelian group $B$ the map
\begin{align*}
\psi_B(\omega)(l_1, l_2, l_3) = \prod_{\sigma \in \mathbb{S}_3}  \omega(l_{\sigma(1)}, l_{\sigma(2)}, l_{\sigma(3)})^{\operatorname{sgn}(\sigma)}, \quad \omega \in C^3(B,\mathbb C^*)
\end{align*}
induces a group homomorphism 
\[
\psi_B: H^3(B, \mathbb{C}^\times) \to \operatorname{Hom}(\wedge^3 B, \mathbb{C}^\times).
\]
We will denote the kernel of the map $\psi_B$ by $H_{\text{sym}}^3(B, \mathbb{C}^\times)$.

\begin{lemma}
We have a group homomorphism $s: H_{\text{sym}}^3(B, \mathbb{C}^\times) \to H^2(B, \widehat{B})$.
\end{lemma}

\begin{proof}
As shown in the proof of Proposition \ref{prop:Obstruction1}, a 3-cocycle $\omega$ belongs to $Z^3_{\text{sym}}(B, \mathbb{C}^\times)$ if and only if there exists a 2-cochain $l \in C^2(B, \mathbb{C}^*)$ such that equation \eqref{eq:BZ2} holds. Observe that if $l' \in C^2(B, \mathbb{C}^*)$ is another 2-cochain satisfying \eqref{eq:BZ2}, then $l/l' \in C^1(B, \widehat{B})$.

Now, by Proposition \ref{Prop: obstruccion O2}, we can define $s(\omega)$ as the cohomology class of
\[
O_2(\omega, l)(b_1 | b_2, b_3) = \frac{\omega(b_1, b_2, b_3) \omega(b_3, b_1, b_2)}{\omega(b_1, b_3, b_2)} \frac{l(b_1, b_3) l(b_2, b_3)}{l(b_1 b_2, b_3)},
\]
where $b_2, b_3 \in B$. This cohomology class is independent of the choice of $l$, as for another $l'$, we have $O_2(\omega, l') = O_2(\omega, l) d^1(l'/l)$.
\end{proof}

\begin{proposition}
Let $(A,q)$ be a pre-metric group. Then:
\begin{enumerate}
    \item The exact sequence \eqref{eq: sequence associ zesting} splits canonically, and $\mathbf{Sez}(A,q) = H^3(\widehat{A}, \mathbb{C}^*) \times H^2(\widehat{A}, A_0)$.
    \item $\ker(PW^2) = \operatorname{Im}(s) \cap H^2_{\text{sym}}(\widehat{A}, A_0)$.
\end{enumerate}
\end{proposition}

\begin{proof}
Given that the twist is trivial on $A_0$, we can consider $c$ as trivial. Consequently, an associative zesting corresponds to a pair $(\lambda, \omega) \in Z^2(\widehat{A}, A_0) \times Z^3(\widehat{A}, \mathbb{C}^*)$. Hence $\mathbf{Sez}(A,q) = H^3(\widehat{A}, \mathbb{C}^*) \times H^2(\widehat{A}, A_0)$.

Now, consider $(\lambda, \omega)$ as an associative zesting. It follows from Proposition \ref{prop:Obstruction1} that $\omega \in Z^3_{\text{sym}}(\widehat{A}, \mathbb{C}^*)$. Furthermore, Proposition \ref{Prop: obstruccion O2} implies that $\overline{\lambda} = s(\omega)$, leading to the conclusion that $\lambda = s(\omega) \in \operatorname{Im}(s) \cap H^2_{\text{sym}}(\widehat{A}, A_0)$. Conversely, if $\lambda \in \operatorname{Im}(s) \cap H^2_{\text{sym}}(\widehat{A}, A_0)$, then, by the same propositions, there exists an $\omega$ such that $(\lambda, \omega)$ defines an associative zesting.
\end{proof}

\begin{remark}\label{remark: algritmo para caso tannakiano}
We have the following algorithmic procedure for computing the braided zesting in the Tannakian case:
\begin{enumerate}
    \item Compute $H^3_{\text{sym}}(\widehat{A}, \mathbb{C}^*)$.
    \item For $\omega\in Z^3_{ab}(\widehat{A},\mathbb C^*)$ compute $l \in C^2(\widehat{A}, \mathbb{C}^*)$ such that equation \eqref{eq:BZ2} holds.
    \item Compute $O_2(\omega, l)$ and check if $O_2(\omega, l)$ is symmetric.
\end{enumerate}
If $O_2(\omega, l)$ is not symmetric, then there is no braided zesting datum of the form $(\lambda, \omega, l)$ for any 3-cocycle cohomologous to $\omega$. If $O_2(\omega, l)$ is symmetric, then $(O_2(\omega, l), \omega, l)$ is a braided zesting datum, and every braided zesting datum is constructed in this way. 
\end{remark}

\subsubsection{Tannakian Zesting by Klein Group}
As an example of Tannakian zesting, we will consider $A = \mathbb{Z}/2 \times \mathbb{Z}/2$. Our interest in this case lies in the fact that we will later see it corresponds to a family of Verlinde modular data.

\begin{proposition}\label{Prop:Tannakian zesting Klein}
For the pre-metric group $(\mathbb{Z}/2 \times \mathbb{Z}/2, 1)$ with the trivial quadratic form, we have that $\mathbf{Sez}_{\text{br}}(\mathbb{Z}/2 \times \mathbb{Z}/2, 1) = H^3_{\text{ab}}(\widehat{A}, \mathbb{C}^*) = \mathbb{Z}/4 \times \mathbb{Z}/4 \times \mathbb{Z}/2$.
\end{proposition}

\begin{proof}
We will apply the procedure outlined in Remark \ref{remark: algritmo para caso tannakiano}.

Step 1: Obtain 3-cocycles. \\
By \cite[Proposition 3.9]{huang2014braided}, $H^3(\Z/2 \times \Z/2,\mathbb{C}^*)=(\Z/2)^{\times 3}$, and a set of representative normalized 3-cocycles is given by
\begin{equation}\label{eq:represntative 3-cocycle of Klein group group}
\omega_{a,b,c}(\vec{x}, \vec{y}, \vec{z}) = (-1)^{ax_1y_1z_1 + cx_2y_1z_1+bx_2y_2z_2},
\end{equation}
where $a,b,c \in \Z/2$ and $(\vec{x},\vec{y},\vec{z}) \in (\Z/2\times\Z/2)^{\times 3}$.

Step 2: Obtain 2-cochains.
We also define a 2-cochain $l_{a,b,c}$ satisfying \eqref{eq:AlternatingProduct} as
\[
l_{a,b,c,\zeta}( \vec{x}, \vec{y} ) = \zeta^{ax_1y_1+cx_2y_1+bx_2y_2},
\]
where $\zeta^2 = -1$.

Step 3: Compute $O_2(-\ | \vec{y}, \vec{z})$. \\
The character $O_2(-\ | \vec{y}, \vec{z})$ is expressed as
\[
O_2 (\vec{x} | \vec{y}, \vec{z}) = (-1)^{c ( x_1y_1z_2 + x_2y_2z_1)},
\]
which defines a 2-cocycle $O_2(w_{a,b,c}) \in H^2(\Z/2 \times \Z/2, \Z/2 \times \Z/2 )$, given by
\[
\lambda_{c}(\vec{y},\vec{z})=(cy_1z_2,cy_2z_1).
\]

This 2-cocycle is symmetric if and only if $c=0$. Since $c=0$, $\lambda$ must be trivial, and thus as we mentioned in Remark \ref{Remark braided and abelian cocyles}, a braided zesting datum is uniquely determined by an abelian 3-cocycle.
\end{proof}

\subsection{Zesting by Klein four group}
As a final example, we will delve into the details of a case that is neither cyclic nor Tannakian. We have again chosen $\mathbb{Z}/2 \times \mathbb{Z}/2$, but this time with $q(\vec{x})=(-1)^{x_1^2+x_2^2}$ and associated bicharacter 
\begin{equation}\label{braiding Dn even}
    c(\vec{x}, \vec{y}) = (-1)^{x_1y_1 + x_2y_2}.
\end{equation}
The primary reason for developing this example in detail is that among the modular categories of the form $\mathcal{C}(\mathfrak{g}, k)$, the only one with a non-cyclic and non-Tannakian pointed subcategory is $\mathcal{C}(\mathfrak{so}(4n), 2s)$. The pointed part is precisely $\mathbb{Z}/2 \times \mathbb{Z}/2$ with a twist given by $\theta(\vec{x}) = (-1)^{x_1^2 + x_2^2}$ for all $\vec{x} = (x_1, x_2) \in \mathbb{Z}/2 \times \mathbb{Z}/2$, which coincides with $\theta_{\vec{x}} = c(\vec{x}, \vec{x})$.

\subsubsection{Symmetric 2-Cocycles of the Klein Four Group}
Let $A_1$, $A_2$, and $B$ be abelian groups. According to \cite{karpilovsky1985projective}, we have $H^2_{\text{Sym}}(A_1 \times A_2, B) \cong H^2_{\text{Sym}}(A_1, B) \oplus H^2_{\text{Sym}}(A_2, B)$. In our case of interest, we have that $H^2_{\text{Sym}}(\mathbb{Z}/2 \times \mathbb{Z}/2, \mathbb{Z}/2 \times \mathbb{Z}/2) = (\mathbb{Z}/2 \times \mathbb{Z}/2)^{\times 2}$, and a representative 2-cocycle for each pair $(\vec{m}, \vec{n}) \in (\mathbb{Z}/2 \times \mathbb{Z}/2)^{\times 2}$ is given by:
\begin{equation}\label{eq: 2-cocycle lambda}
    \lambda_{\vec{m}, \vec{n}}(\vec{x}, \vec{y}) = (m_1 x_1 y_1 + n_1 x_2 y_2, m_2 x_1 y_1 + n_2 x_2 y_2),
\end{equation}
where $\vec{x} = (x_1, x_2)$ and $\vec{y} = (y_1, y_2)$ belong to $\mathbb{Z}/2 \times \mathbb{Z}/2$.

\subsubsection{Computing the Partial Obstructions}
Recall that we are using the bicharacter defined in equation \eqref{braiding Dn even} and the partial obstruction developed in Section \ref{subsecton: sequencial obstructions}. We aim to identify which of the symmetric 2-cocycles $\lambda_{\vec{m}, \vec{n}} \in Z^2 ( \Z/2 \times \Z/2, \Z/2 \times \Z/2)$ and 3-cochains $\omega\in C^3( \Z/2 \times \Z/2, \C^*)$, denoted by $(\lambda_{\vec{m}, \vec{n}}, \omega)$, admit a braided zesting. To do this, we first establish the following proposition:
\begin{proposition} \label{prop:AssocZestingKleinNonTanakian}
The map $\cup_c: H^2_{\text{sym}}(\mathbb{Z}/2 \times \mathbb{Z}/2,\mathbb{Z}/2 \times \mathbb{Z}/2) \to H^4(\mathbb{Z}/2 \times \mathbb{Z}/2, \mathbb{C}^*)$ is trivial. A representative associative zesting datum for each $(\vec{m},\vec{n}) \in (\mathbb{Z}/2 \times \mathbb{Z}/2)^{\times 2}$ is given by
\begin{align}
\lambda_{\vec{m},\vec{n}}(\vec{x}, \vec{y}) &= (m_1 x_1 y_1 + n_1 x_2 y_2, m_2 x_1 y_1 + n_2 x_2 y_2), \label{eq:AssocZestingKlein2Cocycle} \\
\omega_{\vec{m},\vec{n}}(\vec{x}, \vec{y}, \vec{z}) &= \zeta_{4}^{(m_1+m_2)x_1y_1z_1 +(n_1+n_2)x_2y_2z_2 +(\vec{m}\cdot\vec{n})x_1y_1z_2+(\vec{m}\cdot\vec{n})x_2y_2z_1},\label{eq:AssocZestingKlein3Cochain}
\end{align}
where $\zeta_4$ is a fourth root of unity.
\end{proposition}

\begin{proof}

We need to show that the pairs $(\lambda_{\vec{m},\vec{n}}, \omega_{\vec{m},\vec{n}})$ given by equations \eqref{eq:AssocZestingKlein2Cocycle} and \eqref{eq:AssocZestingKlein3Cochain} define an associative zesting datum, that is, $d^3(\omega_{\vec{m},\vec{n}}) = \lambda_{\vec{m},\vec{n}} \wedge_c \lambda_{\vec{m},\vec{n}}$.

Taking the bicharacter defined in equation \eqref{braiding Dn even}, we have
\begin{equation}\label{4 cocycle Dn even n}
  \lambda_{\vec{m}, \vec{n}}\wedge_c \lambda_{\vec{m}, \vec{n}} = (-1)^{(m_1+m_2)x_1y_1z_1w_1 + (n_1+n_2) x_2y_2z_2w_2 + (\vec{m}\cdot \vec{n})x_1y_1z_2w_2+(\vec{m}\cdot \vec{n})x_2y_2z_1w_1}.
\end{equation}
To simplify computations, we define $\alpha_{i,j,k,l} \in C^4(\mathbb{Z}/2 \times \mathbb{Z}/2, \mathbb{C}^*)$ for $i,j,k,l \in \{1,2\}$ as
\begin{equation}\label{4 cofronteras base}
  \alpha_{i,j,k,l}(\vec{x}, \vec{y}, \vec{z}, \vec{w}) = (-1)^{x_iy_jz_kw_l}.
\end{equation}
We also define 3-cochains $\omega_{(i,j,k),\xi} \in C^3(\Z/2 \times \Z/2, \mathbb{C}^*)$ as
\begin{equation}
  \omega_{(i,j,k)}(\vec{x}, \vec{y}, \vec{z}) = \zeta_4^{x_iy_jz_k}.
\end{equation}
Using these definitions, we can rewrite $\lambda_{\vec{m}, \vec{n}}\wedge_c \lambda_{\vec{m}, \vec{n}}$ and $\omega_{(i,j,k)}$ as
\begin{align*}
  \lambda_{\vec{m}, \vec{n}}\wedge_c \lambda_{\vec{m}, \vec{n}} &= \left( \alpha_{1,1,1,1}\right)^{m_1+m_2} \left(\alpha_{2,2,2,2} \right)^{n_1+n_2} \left(\alpha_{1,1,2,2}\alpha_{2,2,1,1}\right)^{\vec{m}\cdot\vec{n}}, \\
  \omega_{\vec{m},\vec{n}} &= \left(\omega_{(1,1,1)
  }\right)^{m_1+m_2} \left(\omega_{(2,2,2)}\right)^{n_1+n_2} \left(\omega_{(1,1,2)}\omega_{(2,2,1)}\right)^{\vec{m}\cdot\vec{n}}.
\end{align*}
To verify that $d^3(\omega_{(\vec{m},\vec{n})}) = \lambda_{\vec{m}, \vec{n}}\wedge_c \lambda_{\vec{m}, \vec{n}}$, it suffices to show that $d^3(\omega_{1,1,1}) = \alpha_{1,1,1,1}$, $d^3(\omega_{2,2,2}) = \alpha_{2,2,2,2}$, etc. For instance, we have
\begin{align*}
  d^3(\omega_{(1,1,2)})(\vec{x}, \vec{y}, \vec{z}, \vec{w}) &= \frac{\zeta_4^{y_1z_1w_2}\zeta_4^{x_1(y_1+z_1)w_2}\zeta_4^{x_1y_1z_2}}{\zeta_4^{(x_1+y_1)z_1w_2}\zeta_4^{x_1y_1(z_2+w_2)}} \\
  &= \zeta_4^{(y_1 - (x_1+y_1))z_1w_2 + (z_2 - (z_2+w_2))x_1y_1 + (y_1+z_1)x_1w_2} \\
  &= \zeta_4^{y_1z_1w_2 - x_1z_1w_2 + x_1y_1z_2 - x_1y_1w_2 + x_1y_1w_2} \\
  &= \zeta_4^{y_1z_1w_2 - x_1z_1w_2 + x_1y_1z_2} \\
  &= \alpha_{1,1,2,2}(\vec{x}, \vec{y}, \vec{z}, \vec{w})
\end{align*}

Note that if $x_1 = y_1 = z_2 = w_2 = 1$, then
\[
  d^3(\omega_{(1,1,2)})(\vec{x}, \vec{y}, \vec{z}, \vec{w}) = \alpha_{1,1,2,2}(\vec{x}, \vec{y}, \vec{z}, \vec{w}).
\]
The other cases follow similarly by case-by-case computation.
\end{proof}

Having established Proposition \ref{prop:AssocZestingKleinNonTanakian}, we can now proceed to describe the equivalence classes of braided zestings and representatives of modular data with $A = (\mathbb{Z}/2)^{\times 2}$ and bicharacter \eqref{braiding Dn even}.

\begin{theorem}\label{Braided zesting non tanakian klein}
Let $(\mathbb{Z}/2 \times \mathbb{Z}/2,q)$ be the pre-metric group with $q(\vec{x})=(-1)^{x_1^2+x_2^2}$.
Then $\operatorname{ker}(PW^2) = \mathbb{Z}/2 \times \mathbb{Z}/2$, that is, the $H^2_{ab}(\widehat{A}, \mathbb{C}^*)$-orbits of equivalence classes of braided zestings of $(\mathbb{Z}/2 \times \mathbb{Z}/2,q)$ are parametrized by $\mathbb{Z}/2 \times \mathbb{Z}/2$. Equivalence classes of braided zesting data can be parametrized by $\mathbb Z/8\times \mathbb Z/8\times \mathbb Z/2$, and representatives for a triple $(m, n,c) \in \mathbb{Z}/4 \times \mathbb{Z}/4\times \mathbb Z/2$ is given by
\begin{align}
    \lambda_{m, n}(\vec{x}, \vec{y}) &= \lambda_{(m, 0), (0, n)}(\vec{x}, \vec{y}) = (m x_1 y_1 \mod 2, n x_2 y_2\mod 2), \label{zesting lambda}\\
    \omega_{(m, n)}(\vec{x}, \vec{y}, \vec{z}) &= \zeta_8^{2(m x_1 y_1 z_1 + n x_2 y_2 z_2)}, \label{zesting omega}\\
    t_{(m, n,c)}(\vec{x}, \vec{y}) &= \zeta_8^{m x_1 y_1 + n x_2  y_2}(-1)^{cx_1y_2}, \label{zesting t}
\end{align}
where $\zeta_8$ is a primitive eighth root of unity.
\end{theorem}

\begin{proof}
Our goal is to determine the conditions under which an associative zesting also admits a braided zesting structure. To achieve this, we follow the procedure outlined in Remark \ref{rmk: procedimiento para calcular braided zestings}.

Consider the pair $(\lambda_{\vec{m}, \vec{n}}, \omega_{(\vec{m},\vec{n})})$, defined in Proposition \ref{prop:AssocZestingKleinNonTanakian}. A straightforward computation shows that $\omega_{(\vec{m},\vec{n})}$ satisfies equation \eqref{eq:AlternatingProduct}.

We then define the 2-cochain $l_{(\vec{m},\vec{n},\vec{\xi}, \vec{\zeta})} \in C^2(\mathbb{Z}/2 \times \mathbb{Z}/2, \mathbb{C}^*)$ as follows:
\begin{equation*} \label{eq:lSinLimitar}
l_{(\vec{m},\vec{n},\vec{\xi}, \vec{\zeta})}(\vec{x}, \vec{y}) = \zeta_8^{(m_1+m_2)x_1 y_1+(n_1+n_2)x_2 y_2+(\vec{m}\cdot \vec{n})x_2 y_1+(\vec{m}\cdot \vec{n})x_1 y_2},
\end{equation*}
where $\zeta_8$ is an eighth root of unity.
 This 2-cochain satisfies equation \eqref{eq:BZ2}.

Next, we apply Proposition \ref{Prop: obstruccion O2} to identify which of the previously defined triples $(\lambda_{\vec{m},\vec{n}}, \omega_{(\vec{m},\vec{n},\vec{\xi})}, l_{(\vec{m},\vec{n},\vec{\xi}, \vec{\zeta})})$ satisfy equation \eqref{eq:BZ3}. To do this, we compute the bicharacter $O_2(- \, | \vec{y},\vec{z})$ as shown below:
\begin{align*}
O_2(\vec{x} | \vec{y}, \vec{z}) &= (\xi_{8}^2)^{(m_1+m_2)2x_1y_1z_1+(n_1+n_2)2x_2y_2z_2+(\vec{m}\cdot\vec{n})x_1y_1 z_2 + x_2y_2z_1+(\vec{m}\cdot\vec{n})x_2y_2 z_1 + x_1y_1z_2} \\
&= (-1)^{(m_1+m_2)x_1y_1z_1 + (n_1+n_2) x_2y_2z_2} (\xi_8^2)^{(\vec{m}\cdot\vec{n})x_1y_1 z_2 + x_2y_2z_1+(\vec{m}\cdot\vec{n})x_2y_2 z_1 + x_1y_1z_2}
\end{align*}

The corresponding 2-cocycle $O_2(\omega) \in Z^2 (\Z/2 \times Z/2 , \Z/2\times \Z/2)$ is then given by:
\begin{equation}
O_2 (\omega_{(\vec{m},\vec{n},\vec{\xi})})(\vec{y}, \vec{z}) = \big( (m_1+m_2)y_1z_1, (n_1+n_2)y_2z_2 \big) +  \big((\vec{m}\cdot \vec{n}) y_1z_2,  (\vec{m}\cdot \vec{n}) y_2z_1\big)
\end{equation}

We find that the cohomology classes match if and only if $m_1+m_2=m_1$ (i.e., $m_2=0$) and $n_1+n_2=n_1$ (i.e., $n_1=0$).
As a result, 2-cocycles that admit a braided zesting structure are of the form:
\begin{equation}
\lambda_{m,n}(\vec{x},\vec{y}) = \lambda_{(m,0),(0,n)}(\vec{x},\vec{y}) = (mx_1y_1, nx_2y_2),
\end{equation}
and the corresponding braided zesting data is given by 
equations \eqref{zesting omega} and \eqref{zesting t}.
\end{proof}

\section{Verlinde Modular Data}\label{section:Verlinde}

Among all modular categories, those derived from the representations of quantum groups at roots of unity are of particular interest. Following \cite{EGNO}, we refer to these as \emph{Verlinde modular categories}. These categories have wide-ranging applications, most notably in low-dimensional topology. In this section, we offer a concise overview of how the modular data for these categories are constructed. The central aim of this section, and a primary focus of this paper, is to provide a comprehensive classification of all braided zesting data in Verlinde modular categories. Specifically, we employ Verlinde modular data to classify braided zesting data up to equivalence and to present the associated zested modular data.

\subsection{Modular Categories from Quantum Groups and their group of invertible objects}

For the normalization of inner products in root systems, we adopt the convention of \cite{BK}. This convention distinguishes two types of inner products: $\langle\langle \alpha, \alpha \rangle\rangle = 2$ for short roots, and $\langle \alpha, \alpha \rangle = 2$ for long roots. These two products are related by $\langle\langle \ , \  \rangle\rangle = m\langle \ , \  \rangle$, where $m = 3$ for $G_2$, $m = 2$ for $B_n, C_n,$ and $F_4$, and $m = 1$ in all other cases. For additional details on this normalization, please refer to the appendix.

Next, we introduce some notation for $\mathcal{C} (\mathfrak{g}, k)$, the modular categories associated with $\mathfrak{g}$ and a positive integer $k > 0$, known as the \emph{level}. For details, see \cite{BK}, \cite{sawin2006closed}, and \cite{sawin2002jones}. The categories $\mathcal{C}(\mathfrak{g}, k)$ are constructed as the semisimplification of the \emph{tilting} representations of $U_q(\mathfrak{g})$, where $q = e^{\frac{\pi i}{m(k + h^\vee)}}$. Here, $h^\vee$, the Coxeter dual number of $\mathfrak{g}$, is defined as $h^\vee = \langle \rho, \Theta \rangle + 1$, with $\Theta$ being the highest root of the root system of $\mathfrak{g}$.

For future reference and reader convenience, Table \ref{convention} contains the notation used for the associated root systems, Lie algebras, and other typical notations found in the physics literature for the associated modular categories (or anyon theories, in physical terminology).

\begin{table}[h!]
\centering
\small
\setlength{\tabcolsep}{12pt} 
\renewcommand{\arraystretch}{1.5} 
\begin{tabular}{>{\centering\arraybackslash}m{1.5cm} >{\centering\arraybackslash}m{2.5cm} >{\centering\arraybackslash}m{3cm} >{\centering\arraybackslash}m{2.5cm}}
\toprule
Root System & $\mathfrak{g}$ & $q=e^{\frac{\pi i}{m(h^{\vee}+k)}}$ & $\mathcal{C}(\mathfrak{g},k)$ \\
\midrule
$A_n$ & $\mathfrak{sl}_{n+1}$ & $e^{\frac{\pi i}{n+1 +k}}$ & $SU(n+1)_k$ \\
\midrule
$E_6$ & $\mathfrak{e}_6$ & $e^{\frac{\pi i}{12+k}}$ & $E_6(k)$ \\
\midrule
$B_n$ & $\mathfrak{so}_{2n+1}$ & $e^{\frac{\pi i}{2(2n-1+k)}}$ & $SO(2n+1)_k$ \\
\midrule
$C_n$ & $\mathfrak{sp}_{n}$ & $e^{\frac{\pi i}{2(n+1+k)}}$ & $Sp(2n)_k$ \\
\midrule
$E_7$ & $\mathfrak{e}_7$ & $e^{\frac{\pi i}{18+k}}$ & $E_7(k)$ \\
\midrule
$D_n$ & $\mathfrak{so}_{2n}$ & $e^{\frac{\pi i}{2n-2+k}}$ & $SO(2n)_k$ \\
\midrule
$E_8$ & $\mathfrak{e}_8$ & $e^{\frac{\pi i}{30+k}}$ & $E_8(k)$ \\
\midrule
$G_2$ & $\mathfrak{g}_2$ & $e^{\frac{\pi i}{3(4+k)}}$ & $G_2(k)$ \\
\midrule
$F_4$ & $\mathfrak{f}_4$ & $e^{\frac{\pi i}{2(9+k)}}$ & $F_4(k)$ \\
\bottomrule
\end{tabular}
\caption{Root of unity for \emph{positive integer levels} and alternative names for $\mathcal{C}(\mathfrak{g}, k)$.}
\label{convention}
\end{table}

\subsection{Modular Data}

We will adopt the notation for root systems as outlined in the Appendix. For a given root system, we will use $P$ to denote the weight lattice, $P_+$ for the dominant Weyl alcove, $\alpha_i$ for a fixed set of simple roots, and $\lambda_i$ for the fundamental weights. The set of isomorphism classes of simple objects in $\mathcal{C}(\mathfrak{g}, k)$ has a well-established parametrization, which can be found in \cite{sawin2002jones}.
\begin{equation}\label{eq:simples}
\Irr(\cC (\mathfrak{g}, k)):= \left\lbrace \lambda \in P_{+} :  \langle \lambda, \Theta \rangle  \leq k\right \rbrace,   
\end{equation}
where $\Theta$  is the highest root of the root system of $\mfg$. 

By applying the condition specified in equation \eqref{eq:simples}, and using the formulas for the fundamental weights, highest roots, and the sum of positive roots found in the appendix, we can define the set of labels for the modular data of Verlinde categories as sets of tuples of non-negative integers that satisfy explicit inequalities with respect to a linear combination of the coefficients. These conditions are presented in Table \ref{irreducibles}.

\begin{table}[h!]
\centering
\begin{tabular}{ll}
\toprule
$\mathcal{C}(\mathfrak{g},k)$ & $\text{Irr}(\mathcal{C}(\mathfrak{g},k))$ \\
\midrule
$SU(n+1)_k$ & $\{(a_1, \ldots, a_n) : \sum_{j=1}^n a_j \leq k\}$ \\
$E_6(k)$ & $\{(a_1, \ldots, a_6) : a_1 + 2a_2 + 2a_3 + 3a_4 + 2a_5 + a_6 \leq k\}$ \\
$SO(2n+1)_k$ & $\{(a_1, \ldots, a_n) : a_1 + 2\sum_{j=2}^{n-1} a_j + a_n \leq k\}$ \\
$Sp(2n)_k$ & $\{(a_1, \ldots, a_n) : \sum_{j=1}^n a_j \leq k\}$ \\
$E_7(k)$ & $\{(a_1, \ldots, a_7) : 2a_1 + 2a_2 + 3a_3 + 4a_4 + 3a_5 + 2a_6 + a_7 \leq k\}$ \\
$SO(2n)_k$ & $\{(a_1, \ldots, a_n) : a_1 + 2\sum_{j=1}^{l-2} a_j + a_{n-1} + a_n \leq k\}$ \\
$E_8(k)$ & $\{(a_1, \ldots, a_8) : 2a_1 + 3a_2 + 4a_3 + 6a_4 + 5a_5 + 4a_6 + 3a_7 + 2a_8 \leq k\}$ \\
$G_2(k)$ & $\{(a_1, a_2) : a_1 + 2a_2 \leq k\}$ \\
$F_4(k)$ & $\{(a_1, \ldots, a_4) : 2a_1 + 3a_2 + 2a_3 + a_4 \leq k\}$ \\
\bottomrule
\end{tabular}
\caption{Labels of Verlinde modular data. Conditions on the tuples of non-negative integers}
\label{irreducibles}
\end{table}

We also have explicit formulas for the $S$ and $T$ matrices, which can be found in \cite[Theorem 3.3.20]{BK}:
\begin{enumerate}
    \item The $S$-matrix is given by:
    \begin{equation}
        s_{\lambda,\mu} = \frac{\sum_{w \in W} \operatorname{sgn}(w) q^{2 \left\langle\langle w(\lambda + \rho), \mu + \rho \right\rangle\rangle}}{\sum_{w \in W} \operatorname{sgn}(w) q^{2 \left\langle\langle w(\rho), \rho \right\rangle\rangle}}
    \end{equation}
    Here, $W$ denotes the Weyl group, and $\operatorname{sgn}(w)$ is the sign of $w \in W$.
    
    \item The $T$-matrix is expressed as:
    \begin{equation} \label{tmatrix quantum}
        T_{\lambda,\lambda} = \theta_{\lambda} = q^{\left\langle \langle \lambda, \lambda + 2\rho \right\rangle \rangle}.
    \end{equation}
\end{enumerate}

\begin{example}\label{Example: modular data SU(2)}
Let us consider the simplest case, that is $SU(2)_k$. The simple objects $\Irr(SU(2)_k)$ are parameterized by non-negative integers $0,1,\cdots,k$ and the $S$ and $T$ matrices and fusion coefficients are given by
\begin{align}
    s_{ab} &= \sqrt{\frac{2}{k + 2}} \sin\left(\pi \frac{(a + 1)(b + 1)}{k + 2}\right), \\
    T_{aa} &= \exp \left( \frac{\pi i a^2}{2(k+2)} + \frac{\pi ia}{k+2} \right), \\
    N_{ab}^c &= 
    \begin{cases}
        1 & \text{if } c \equiv a+b \pmod{2} \text{ and } |a-b| \leq c \leq \min\{a+b, 2k-a-b\}, \\
        0 & \text{otherwise}.
    \end{cases}
\end{align}
\end{example}

\subsection{Invertible objects and grading of Verlinde modular data}

For any simple Lie algebra $\mathfrak{g}$, there is a corresponding simply-connected Lie group $G$. The center of this group, $Z(G)$, has a one-to-one correspondence with the invertible objects of the modular category $\mathcal{C}(\mathfrak{g}, k)$. This holds true for all cases except when $\mathfrak{g} = E_8$ and $k = 2$. For further details, refer to \cite{fuchs1991simple} and \cite{sawin2006closed}. Interestingly, for $E_8( 2)$, Sawin proved in \cite{sawin2006closed} that the category corresponds to an Ising category, despite the fact that simply-connected Lie groups of type $E_8$ have a trivial center.

The centers of the corresponding simply-connected Lie groups are described in Table \ref{table:center_of_groups}.

\begin{table}[h!]
\centering
\small
\begin{tabular}{llllll}
\toprule
\textbf{Root System} & $A_n$ & $B_n, C_n, E_7$ & $D_n$ & $E_6$ & $E_8, F_4, G_2$ \\
\midrule
$\operatorname{Z}(G)$ & $\mathbb{Z}/(n+1)$ & $\mathbb{Z}/2$ & $\begin{array}{l} \mathbb{Z}/4 \text{ if } n \text{ is odd} \\ \mathbb{Z}/2 \times \mathbb{Z}/2 \text{ if } n \text{ is even} \end{array}$ & $\mathbb{Z}/3$ & 1 \\
\bottomrule
\end{tabular}
\caption{Center of the corresponding Lie groups for different root systems}
\label{table:center_of_groups}
\end{table}
As a result, the group of invertible objects, denoted by $A$, forms a cyclic group except for $SO(2n)_k$ when $n$ is even. In cyclic cases, the coefficients of the $T$-matrix on $A$ can be determined by evaluating a fixed generator of the group. For $SO(2n)_k$ with $n$ even, these $T$-matrix coefficients are computed using two non-trivial elements of the group.

Since our focus is on the classification of braided zestings, we will exclude modular categories $\mathcal{C}(\mathfrak{g}, k)$ with trivial invertible objects. Importantly, from this point forward, we will \textbf{not} consider cases where $\mathfrak{g}$ is of type $E_8$, $F_4$, or $G_2$.

The following result, which is a direct consequence of \cite[Theorem 3]{sawin2002jones} and \cite[Remark 3]{sawin2002jones}, is crucial for describing the labels of invertible objects and the associated coefficients in the $S$ and $T$ matrices.

\begin{proposition}\label{propo: invertibles}
    For every Lie algebra $\mfg \neq E_8$, the invertible objects are of the form $k\lambda_{i}$ where $\lambda_{i}$ is a fundamental weight in the Weyl alcove such that the associated root $\alpha_i$ is long.
\end{proposition}
\qed

For each type of modular data where $A$ is a cyclic group, Table \ref{Table: generators and twist Cyclic case} lists the multiples of the fundamental weights that correspond to the invertibles and a generator, along with its twist coefficients on that generator $A$.

\begin{table}[h!]
\centering
\begin{tabular}{lll}
\toprule
$\mathcal{C}(\mathfrak{g}, k)$ & $A = \langle g \rangle$ & $\theta_g$ \\
\midrule
$SU(n+1)_k$ & $\{0, k\lambda_1, \ldots, k\lambda_n\} = \langle k \lambda_1\rangle$ & $e^{\frac{\pi i kn}{n+1}}$ \\
$E_6(k)$ & $\{ 0, k\lambda_1, k \lambda_6\} = \langle k\lambda_1 \rangle$ & $(-1)^k e^{\frac{k\pi i}{3}}$ \\
$SO(2n+1)_k$ & $\{ 0, k\lambda_1\} = \langle k\lambda_1 \rangle$ & $(-1)^k$ \\
$Sp(2n)_k$ & $\{ 0, k\lambda_n\} = \langle k\lambda_n \rangle$ & $i^{kn}$ \\
$E_7(k)$ & $\{ 0, k\lambda_7\} = \langle k\lambda_7 \rangle$ & $(-1)^k i^k$ \\
$SO(2n)_k$ with $n=2s+1$ odd & $\{0, k\lambda_1, k \lambda_{n-1}, k\lambda_n \} = \langle k \lambda_n \rangle$ & $i^{ks} e^{\pi k i/4}$ \\
\bottomrule
\end{tabular}
\caption{Generators of invertible objects and their twist coefficients}
\label{Table: generators and twist Cyclic case}
\end{table}

When $A$ is a Klein group, the twist coefficients for fixed generators are provided in Table \ref{twistklein}.

\begin{table}[h!]
    \centering
    \begin{tabular}{llll}
    \toprule
    $\mathcal{C}(\mathfrak{g}, k)$ & $A = \langle g_1, g_2 \rangle$ & $\theta_{g_1}$ & $\theta_{g_2}$ \\
    \midrule
    $SO(2n)_k$ ($n = 2s$ even) & $\{0, k\lambda_1, k\lambda_{n-1}, k\lambda_n\} = \langle k\lambda_{n-1}, k\lambda_n \rangle$ & $i^{ks}$ & $i^{ks}$ \\
    \bottomrule
    \end{tabular}
    \caption{Twist coefficients for Klein groups}
    \label{twistklein}
\end{table}

In the following result, we provide formulas for the universal grading of a Verlinde modular datum $\mathcal{C}(\mathfrak{g}, k)$ that has a non-trivial group of invertible objects, excluding the case $\mathfrak{g} = E_8$.

Firstly, for a cyclic group $A$ of order $N$, we establish an identification $\widehat{A} = A$ by fixing both a generator and a corresponding primitive root of unity $e^{\frac{2\pi i}{N}}$. This allows us to define an explicit $\mathbb{Z}/N$ grading. Similarly, in the case of $SO(2n)_k$ with $n$ even, we fix two generators to establish an identification $A=\widehat{A}$  to provide an explicit $\mathbb{Z}/2 \times \mathbb{Z}/2$ grading.

\begin{theorem} \label{grading fixed generators}
Let $\cC(\mfg,k)$ be a Verlinde modular category with non-trivial invertible objects. Then the universal gradings are presented in Table \ref{universalgrading} and Table \ref{universalgradingklein}.
\end{theorem}

\begin{table}[h!]
    \centering
    \small
    \begin{tabular}{llll}
    \toprule
    $\mathcal{C}(\mathfrak{g}, k)$ & $g$ & $X$ & $|X| \mod N$ \\
    \midrule
    $SU(n+1)_k$ & $k\lambda_1$ & $(a_1, \ldots, a_n)$ & $-\sum_{j=1}^n ja_j \mod (n+1)$ \\
    $E_6(k)$ & $k\lambda_1$ & $(a_1, \ldots, a_6)$ & $a_1 + 2a_3 + a_5 + 2a_6 \mod 3$ \\
    $SO(2n+1)_k$ & $k\lambda_1$ & $(a_1, \ldots, a_n)$ & $a_n \mod 2$ \\
    $Sp(2n)_k$ & $k\lambda_n$ & $(a_1, \ldots, a_n)$ & $\sum_{j=1}^n ja_j \mod 2$ \\
    $E_7(k)$ & $k\lambda_7$ & $(a_1, \ldots, a_7)$ & $a_2 + a_5 + a_7 \mod 2$ \\
    $SO(2n)_k$ (with $n$ odd) & $k\lambda_n$ & $(a_1, \ldots, a_n)$ & $\sum_{j=1}^{n-2} 2a_j + (n-2) a_{n-1} + n a_n \mod 4$ \\
    \bottomrule
    \end{tabular}
    \caption{Universal grading for categories with cyclic invertible group }
    \label{universalgrading}
\end{table}

\begin{table}[h!]
    \centering
    \tiny
    \begin{tabular}{llll}
    \toprule
    $g_1$ & $g_2$ & $X$ & $|X| \in \mathbb{Z}/2 \times \mathbb{Z}/2$ \\
    \midrule
    $k\lambda_{n-1}$ & $k\lambda_n$ & $(a_1, \ldots, a_n)$ & $\left( \sum_{j=1}^{n-2} ja_j + sa_{n-1} + (s-1)a_n \mod 2, \sum_{j=1}^{n-2} ja_j + (s-1)a_{n-1} + sa_n \mod 2 \right)$ \\
    \bottomrule
    \end{tabular}
    \caption{Universal grading for $SO(2n)$ and $n$ even.}
    \label{universalgradingklein}
\end{table}

\begin{proof}
Let $A$ denote the group of invertible objects for a given modular datum $M$. According to Proposition \ref{grading modular data}, $M$ is universally graded over $\widehat{A}$. In the context of Verlinde categories, $A$ is either a cyclic group or $\mathbb{Z}/2 \times \mathbb{Z}/2$. In either case, we can identify $A$ with $\widehat{A}$ using the primitive $N$-th root of unity $e^{2\pi i/N}$ where $N$ is the exponent of $A$.

For Verlinde modular categories $\mathcal{C}(\mathfrak{g}, k)$, Lemma 4 from \cite{sawin2002jones} states that if $k\lambda_i$ corresponds to an invertible object—then $\lambda_i$ is a fundamental weight with an associated long simple root $\alpha_i$—then for any other dominant weight $X$ in $\Irr(\mathcal{C}(\mathfrak{g}, k))$, the following equation holds:
\[
\lVert X \rVert(k\lambda_i) = e^{2\pi i \langle \lambda_i, X \rangle}.
\]

By applying the above formula and conducting a case-by-case analysis, we arrive at the results summarized in Table \ref{universalgrading}.
\end{proof}

\subsection{Conditions for Factorizing the Pointed Part of a Verlinde Category}

If $\mathcal{B}$ is a braided fusion category and $\mathcal{D} \subseteq \mathcal{B}$ is a subcategory, the centralizer subcategory is defined as the full subcategory $C_{\mathcal{B}}(\mathcal{D}) = \{ X \in \mathcal{B} : c_{Y,X} \circ c_{X,Y} = \text{id}_{X \otimes Y}, \text{ for all } Y \in \mathcal{D} \}$. Michael Muger proved in \cite{MR1990929} that if $\mathcal{B}$ is a modular category and $\mathcal{D} \subseteq \mathcal{B}$ is a fusion subcategory, then $\mathcal{B} =  C_{\mathcal{B}}(\mathcal{D}) \boxtimes \mathcal{D}$ as braided categories if and only if $\mathcal{D}$ is also modular.

In the case of zestings of modular categories, the centralizer of $\mathcal{B}_{\text{pt}}$, the maximal pointed subcategory, plays a significant role. Specifically, if $\mathcal{B}$ is a modular category, it is universally graded over $\widehat{A}$, where $A$ is the group of isomorphism classes of $\mathcal{B}_{\text{pt}}$. The invertible objects in the trivial component of the grading correspond precisely to the simples in $C_{\mathcal{B}}(\mathcal{B}_{\text{pt}})$. In other words, $\mathcal{B}$ can be factorized as $C_{\mathcal{B}}(\mathcal{B}_{\text{pt}}) \boxtimes \mathcal{B}_{\text{pt}}$ if and only if $\mathcal{B}_{\text{pt}}$ is modular, or equivalently, if the only invertible object in the trivial component is the unit object.

In this section, we will study necessary and sufficient conditions for Verlinde categories to be factorized in the form $C_{\mathcal{B}}(\mathcal{B}_{\text{pt}}) \boxtimes \mathcal{B}_{\text{pt}}$. This is important for the classification of zestings. If the category can be factorized, any braided or associative zesting datum corresponds simply to an element in $H^3_{\text{ab}}(\widehat{A}, \mathbb{C}^*)$ or $H^3(\widehat{A}, \mathbb{C}^*)$, respectively.  Our aim in this section is to subsequently discard those Verlinde categories that can be factorized for the classification of braided zestings.

\begin{theorem}\label{Trivial graded invertible}
Let $\mathcal{C}(\mathfrak{g}, k)$ be a Verlinde category. Table \ref{Table:Admissible levels} provides  necessary and sufficient conditions on the level $k$ for the pointed subcategory $\mathcal{C}(\mathfrak{g}, k)_{\text{pt}}$ to be degenerate. If a level $k$ does not satisfy the conditions in Table \ref{Table:Admissible levels}, then $\mathcal{C}(\mathfrak{g}, k) =  C_{\mathcal{C}(\mathfrak{g}, k)}(\mathcal{C}(\mathfrak{g}, k)_{\text{pt}}) \boxtimes \mathcal{C}(\mathfrak{g}, k)_{\text{pt}} $ as braided categories.

Additionally, Table \ref{invertiblegrade0} provides the generators of the invertible objects in the trivial component of the universal grading. Equivalently, Table \ref{invertiblegrade0} lists generators of the simple objects in $C_{\mathcal{C}(\mathfrak{g}, k)}(\mathcal{C}(\mathfrak{g}, k)_{\text{pt}}) \cap \mathcal{C}(\mathfrak{g}, k)_{\text{pt}}$.
\end{theorem}

\begin{table}[h!]
\centering
\begin{tabular}{ll}
\toprule
$\cC(\mfg,k)$ & Condition \\
\midrule
$SU(n+1)_k$ & $(n+1, k) \neq 1$ \\
$E_6(k)$ & $k \in 3\mathbb{Z}$ \\
$SO(2n+1)_k$ & for all $k$ \\
$Sp(2n)_k$ & $k \in 2\mathbb{Z}$ \\
$E_7(k)$ & $k \in 2\mathbb{Z}$ \\
$SO(2n)_k, n = 2s + 1$ & $k \in 2\mathbb{Z}$ \\
$SO(2n)_k, n = 2s$ & $k \in 2\mathbb{Z}$ \\
\bottomrule
\end{tabular}
\caption{Conditions for $\mathcal{C}(\mathfrak{g}, k)_{\text{pt}}$ to be degenerate.}
\label{Table:Admissible levels}
\end{table}

\begin{table}[h!]
\centering
\footnotesize
\begin{tabular}{llll}
\toprule
$\mathcal{C}(\mathfrak{g}, k)$ & $g$ & $A_0$ & $h = g^a$ \\
\midrule
$SU(n+1)_k$ & $k\lambda_1$ & $\{ k\lambda_j : kj \in (n+1)\mathbb{Z} \}$ & $g^{\frac{n+1}{\gcd(n+1, k)}}$ \\
\midrule
$E_6(k)$ & $k\lambda_1$ & $\langle k\lambda_1 \rangle$ & $g$ \\
\midrule
$SO(2n+1)_k$ & $k\lambda_1$ & $\langle k\lambda_1 \rangle$ & $g$ \\
\midrule
$Sp(2n)_k$ & $k\lambda_n$ & $\langle k\lambda_n \rangle$ & $g$ \\
\midrule
$E_7(k)$ & $k\lambda_7$ & $\langle k\lambda_7 \rangle$ if $k \in 2\mathbb{Z}$ & $g$ \\
\midrule
\multirow{2}{*}{$SO(2n)_k, n=2s+1$ odd} & \multirow{2}{*}{$k\lambda_n$} & 
\begin{tabular}{@{}l@{}}
$\{0, k\lambda_1\}$ if $k \notin 4\mathbb{Z}$ \\
$\{0, k\lambda_1, k\lambda_{n-1}, k\lambda_n \}$ if $k \in 4\mathbb{Z}$ 
\end{tabular} & 
\begin{tabular}{@{}l@{}}
$g^2 = k\lambda_1$ if $k \notin 2\mathbb{Z}$ \\
$g$ if $k \in 4\mathbb{Z}$
\end{tabular} \\
\midrule
$SO(2n)_k, n$ even & $k\lambda_{n-1}, k\lambda_n$ & $A = A_0 = \langle k\lambda_{n-1}, k\lambda_n \rangle$ & \\
\bottomrule
\end{tabular}
\caption{The element $g$ is a generator of $A$, and $h$ a generator of $A_0$.}
\label{invertiblegrade0}
\end{table}

\begin{proof}
    The result follows from direct computation, utilizing the results from Proposition \ref{grading fixed generators}.
\end{proof}

\subsection{Zestings of Verlinde Categories and Their Zested Modular Data}\label{ExamplesVerlinde}

In this final subsection, we present the main result of the paper: the classification of braided zestings for Verlinde categories $\mathcal{C}(\mathfrak{g}, k)$.

According to Theorem \ref{Trivial graded invertible}, if a category $\mathcal{C}(\mathfrak{g}, k)$ does not meet the conditions outlined in Table \ref{Table:Admissible levels}, its associative and braided zestings are classified by the groups $H^3(\widehat{A}, \mathbb{C}^*)$ and $H^3_{\text{ab}}(\widehat{A}, \mathbb{C}^*)$, respectively. Here, $A$ is the group of invertible objects in $\mathcal{C}(\mathfrak{g}, k)$. These groups $A$ are either cyclic or $\mathbb{Z}/2 \times \mathbb{Z}/2$, as indicated in Table \ref{table:center_of_groups}.

For a cyclic group, the abelian cohomology $H^3_{\text{ab}}(\mathbb{Z}/N, \mathbb{C}^*)$ is isomorphic to the set $\{ \nu \in \mathbb{C}^* : \nu^{N^2} = \nu^{2N} = 1 \}$, as was described in equation \eqref{eq: abelian 3 cohomology of cyclic groups}.

For the group $\mathbb{Z}/2 \times \mathbb{Z}/2$, we saw in Proposition \ref{prop:abelian3CocycleOfKleinGroup} that $H^3_{\text{ab}}(\mathbb{Z}/2 \times \mathbb{Z}/2, \mathbb{C}^*) = \mathbb{Z}/4 \times \mathbb{Z}/4 \times \mathbb{Z}/2$, and a set of 3-cocycles representing the cohomology classes was presented in equation \eqref{eq: abelian 3-cocycles of klein}.

In rest of this section, we will focus on describing the braided zestings of $\mathcal{C}(\mathfrak{g}, k)$ that satisfy the conditions in Table \ref{Table:Admissible levels}. The zestings will be  classified differently depending on whether $A$ is cyclic or not.

\subsubsection{Cyclic Zesting}

Consider a pre-metric group $(A, q)$ where $A$ is a cyclic group of order $N$ and $A_0$ has order $m$. Let $g$ be a generator of $A$, and choose a primitive $N$-th root of unity $q$ and $\zeta \in \mathbb{C}^*$ such that $\zeta^2 = q$.

Firstly, define $\epsilon(a)$ as:
\begin{align*}
    \epsilon(a) = \begin{cases} 
    1 & \text{if } q(g^{N/m}) = -1, \\
    0 & \text{if } q(g^{N/m}) = 1 
    \end{cases}
\end{align*}

Associative zesting classes are characterized by pairs $(a, b) \in \mathbb{Z}/N \times \mathbb{Z}/m$. For such a pair $(a, b)$, the zesting parameters $(\lambda_a, \lambda_b)$ are:

\begin{equation}
\begin{aligned}
   \lambda_a(i, j) &= \begin{cases} 
   1 & \text{if } i + j < N, \\
   g^a & \text{if } i + j \ge N 
   \end{cases}, \quad
   \lambda_b(i, j, k) &= \begin{cases} 
   1 & \text{if } i + j < N, \\
   \zeta^{k (\epsilon(a) + 2b)} & \text{if } i + j \ge N 
   \end{cases}
\end{aligned}
\end{equation}

A pair $(a, b)$ admits a braided zesting structure if and only if 
\begin{equation}\label{eq: cyclic braided zesting condition}
 a \equiv \epsilon(a) + 2b \mod N.    
\end{equation}
In this case, the equivalence classes of braided zesting data are parameterized by triples $(a, b, \gamma)$ where $\gamma \in \mathbb{C}^*$ satisfies $\gamma^N = \zeta^{- (\epsilon(a) + 2b)}$. For a specific triple $(a, b, \gamma)$, the additional datum $t_\gamma$ is:
\begin{align}
   t_\gamma(i, j) &= \gamma^{-ij}.
\end{align}

In the case where the modular data originates from a modular category that is universally graded by a cyclic group, it follows from Theorem \ref{th:new_modular_data} that the zested modular data is given by
\begin{align}
  T^\lambda_{X_i, X_i} &= \gamma^{i^2} T_{X_i, X_i}, \quad 0 \leq i < N, \quad X_i \in B_i \label{eq:cyclic T zested}, \\
  s^\lambda_{X_i, Y_j} &= \gamma^{-2ij} s_{X_i, Y_j}, \quad 0 \leq i, j < N, \quad X_i \in \cB_i, \quad Y_j \in \cB_j \label{eq:cyclic S zested}.
\end{align}

For more details, refer to \cite[Propositions 6.3 and 6.4]{zesting}.

\begin{theorem}\label{th: zesting velinde category cyclic}
Let $\mathcal{C}(\mathfrak{g}, k)$ be a Verlinde category as listed in Table \ref{Table:Admissible levels}, excluding the case of $SO(2n)$ for even $n$. The parameters $(a,b)\in \mathbb{Z}/N \times \mathbb{Z}/m$ and $\epsilon_a$, in the cyclic  associative zestings  are detailed in Table \ref{braidedzest}. The last column of Table \ref{braidedzest} contains the pairs $(a,b)$ that admit a braided zesting structure.
\end{theorem}

\begin{table}[h!]
\centering
\footnotesize
\begin{tabular}{lllll}
\toprule
$\mathcal{C}(\mathfrak{g},k)$ & $m$ & $N$ & $\epsilon_a$ & $(a,b): a\frac{N}{m}=\epsilon_a+2b \mod N$ \\
\midrule
\multirow{2}{*}{$SU(n+1)_k$} & $(n+1, k)$ & $n+1$ & $1$ if: 
\begin{tabular}{@{}l@{}}
$(n+1, k)$ is even, \\
$\frac{n+1}{(n+1, k)}, a$ and \\
$\frac{k}{(n+1, k)}$ are odd
\end{tabular} & $a\frac{n+1}{(n+1,k)}=1+2b \mod (n+1)$ \\
& & & 0 else & $a\frac{n+1}{(n+1,k)}=2b \mod (n+1)$ \\
\midrule
$E_6(k)$ & 3 & 3 & 0 & $(0,0), (1,2), (2,1)$ \\
\midrule
\multirow{2}{*}{$SO(2n+1)_k$} & 2 & 2 & 1 if $k$ is odd & $(1,0), (1,1)$ \\
& & & 0 & $(0,0)$ \\
\midrule
\multirow{2}{*}{$Sp(2n)_k$} & 2 & 2 & 1 if $n$ is odd & $(1,0), (1,1)$ \\
& & & 0 & $(0,0)$ \\
\midrule
\multirow{2}{*}{$E_7(k)$} & 2 & 2 & 1 if $k=2s$ with $s$ odd & $(1,0), (1,1)$ \\
& & & 0 & $(0,0)$ \\
\midrule
\multirow{4}{*}{$\begin{array}{@{}l@{}} SO(2n)_k \\ n=2s+1 \end{array}$} & $2$ if $k\notin 4\mathbb{Z}$ & 4 & 0 & $(0,0), (0,2), (1,1), (1,3)$ \\
& $4$ if $k\in 4\mathbb{Z}$ & 4 & 
\begin{tabular}{@{}l@{}}
1 if $k=4j$ with $j,a$ odd \\
0 else
\end{tabular} & 
\begin{tabular}{@{}l@{}}
$(1,0), (1,2), (3,1), (3,3)$ \\
$(0,0), (0,2), (2,1), (2,3)$
\end{tabular} \\
\bottomrule
\end{tabular}
\caption{Parametrization of cyclic associative and braided zestings of $\mathcal{C}(\mathfrak{g}, k)$}
\label{braidedzest}
\end{table}

\begin{proof}
Regarding the verification of the results in Table~\ref{braidedzest}, we will present only the details for $SU(n+1)_k$ and $SO(2n)_k$. Details for the other cases can be calculated in a similar and more straightforward manner.

For $SU(n+1)_k$, we refer to Table~\ref{invertiblegrade0} for the generators of the invertible objects in $A$ and $A_0$, and to Table~\ref{Table: generators and twist Cyclic case} for the value of the twist for the generator. We take $h = g^{\frac{n+1}{\gcd(n+1, k)}}$ as the generator of $A_0$. Using equation~\eqref{tmatrix quantum}, we find that
\begin{align*}
    \theta_{h} &= \left( \theta_g \right)^{\left( \frac{n+1}{\gcd(n+1, k)} \right)^2} \\
    &= \exp\left( \frac{\pi i k n (n+1)^2}{\gcd(n+1, k)^2 (n+1)} \right) \\
    &= \exp\left( \frac{\pi i k n (n+1)}{\gcd(n+1, k)^2} \right) \\
    &= \left(-1 \right)^{\frac{n+1}{\gcd(n+1, k)} \frac{k}{\gcd(n+1, k)} n} 
\end{align*}

Note that if $n+1$ is odd, then $\frac{n+1}{\gcd(n+1, k)} \frac{k}{\gcd(n+1, k)} n$ is even, and thus $\theta_{h} = 1$.

If $n+1$ is even, $\theta_{h} = 1$ if either $\frac{n+1}{\gcd(n+1, k)}$ or $\frac{k}{\gcd(n+1, k)}$ is even. If both are odd, then $\gcd(n+1, k)$ must be even, and we have
\[
    \theta_{h} = \begin{cases}
        -1 & \text{if } \gcd(n+1, k) \text{ is even, and both } \frac{n+1}{\gcd(n+1, k)} \text{ and } \frac{k}{\gcd(n+1, k)} \text{ are odd}, \\
        1 & \text{otherwise}
    \end{cases}
\]
and therefore, if $a \in \mathbb{Z}$,
\[
    \epsilon_a = \begin{cases}
        1 & \text{if } \gcd(n+1, k) \text{ is even, and } \frac{n+1}{\gcd(n+1, k)}, a, \text{ and } \frac{k}{\gcd(n+1, k)} \text{ are odd}, \\
        0 & \text{otherwise}
    \end{cases}
\]

To confirm the data in Table~\ref{braidedzest}, one can directly apply equation~\ref{eq: cyclic braided zesting condition}. Specifically, for $SU(n+1)_k$, the information in Table~\ref{braidedzest} directly follows from equation~\ref{eq: cyclic braided zesting condition}. For other cases, applying equation~\ref{eq: cyclic braided zesting condition} is sufficient once the results in Tables~\ref{invertiblegrade0} and~\ref{Table: generators and twist Cyclic case} are established.

For $SO(2n)_k$, there are two cases to discuss:

When $k$ is a multiple of 4, the pairs $(a, b) \in \mathbb{Z}/4 \times \mathbb{Z}/4$ that define associative zestings $(\lambda_a, \lambda_b)$ which also admit braided zestings must satisfy one of the following conditions:

\[
    a \equiv 1 + 2b \mod 4, \quad \text{with } a \text{ odd},
\]
\[
    a \equiv 2b \mod 4, \quad \text{with } a \text{ even},
\]
yielding the solutions $(1,0), (1,2), (3,1), (3,3)$ for the first equation and $(0,0), (0,2),$ $(2,1), (2,3)$ for the second.

When $k \notin 4\mathbb{Z}$, pairs $(a, b) \in \mathbb{Z}/2 \times \mathbb{Z}/4$ allowing for associative zestings $(\lambda_a, \lambda_b)$ that also admit braided zestings must meet $2a \equiv 2b \mod 4$, yielding the solutions $(0,0), (0,2), (1,1), (1,3)$.

\end{proof}

Finally, we derive explicit formulas for the new modular data of the zested categories using the universal grading formulas outlined in Proposition~\ref{grading fixed generators}, along with the formulas \eqref{eq:cyclic S zested} and \eqref{eq:cyclic T zested}.

\begin{theorem} \label{NDatosModularesCyclyc}
Let $\mathcal{C}(\mathfrak{g}, k)$ be a Verlinde category listed in Table~\ref{Table:Admissible levels}, with the exception of $SO(2n)$ when $n$ is even. The modular data of the braided zestings presented in Table~\ref{braidedzest} takes the form
\begin{align}
\label{T-matrix cyclic general}
T_{X,X}^\lambda &= \gamma^{\alpha_X} T_{X,X}, \\
\label{S-matrix cyclic general}
s_{X,Y}^\lambda &= \gamma^{-2\beta_{X,Y}} s_{X,Y},
\end{align}
where the coefficients $\alpha_X$ and $\beta_{X,Y}$ are presented in Tables~\ref{alpha_x} and~\ref{beta xy}, respectively.
\end{theorem}

\begin{table}[h!]
\centering
\small
\label{your-label-here}
\begin{tabular}{lll}
\toprule
$\mathcal{C}(\mathfrak{g},k)$ & $X$ & $\alpha_X$ \\
\midrule
\multirow{2}{*}{\textit{SU(n+1)$_k$}} & \multirow{2}{*}{$(a_1, \ldots, a_n)$} & $- \left(\sum_{j=1}^n ja_j\right)^2$ \\
& & $\mod n+1$ \\
\midrule
\multirow{2}{*}{\textit{E$_6$(k)}} & \multirow{2}{*}{$(a_1, \ldots, a_6)$} & $- (a_1 + 2a_3 + a_5 + 2a_6)^2$ \\
& & $\mod 3$ \\
\midrule
\multirow{2}{*}{\textit{SO(2n+1)$_k$}} & \multirow{2}{*}{$(a_1, \ldots, a_n)$} & $- a_n^2$ \\
& & $\mod 2$ \\
\midrule
\multirow{2}{*}{\textit{Sp(2n)$_k$}} & \multirow{2}{*}{$(a_1, \ldots, a_n)$} & $- \left(\sum_{j=1}^n j a_j\right)^2$ \\
& & $\mod 2$ \\
\midrule
\multirow{2}{*}{\textit{E$_7$(k)}} & \multirow{2}{*}{$(a_1, \ldots, a_7)$} & $(a_2 + a_5 + a_7)^2$ \\
& & $\mod 2$ \\
\midrule
\multirow{2}{*}{\textit{SO(2n)$_k$, $n = 2s + 1$}} & \multirow{2}{*}{$(a_1, \ldots, a_n)$} & $- \left(\sum_{j=1}^{n-2} 2a_j + (n-2) a_{n-1} + n a_n\right)^2$ \\
& & $\mod 4$ \\
\bottomrule
\end{tabular}
\caption{Coefficient $\alpha_X$ in equation~\eqref{T-matrix cyclic general}}
\label{alpha_x}

\end{table}

\begin{table}[h!]
\centering
\small 
\begin{tabular}{lll}
\toprule
$\mathcal{C}(\mathfrak{g},k)$ & $X, Y$ & $\beta_{X,Y}$ \\
\midrule
\multirow{2}{*}{$SU(n+1)_k$} & $(a_1, \ldots, a_n)$ & \\
                               & $(b_1, \ldots, b_n)$ & $\left(\sum_{j=1}^n ja_j \right)\left(\sum_{j=1}^n jb_j \right) \mod n+1$ \\
\midrule
\multirow{2}{*}{$E_6(k)$} & $(a_1, \ldots, a_6)$ & \\
                            & $(b_1, \ldots, b_6)$ & $(a_1+2a_3+a_5+2a_6)(b_1+2b_3+b_5+2b_6) \mod 3$ \\
\midrule
\multirow{2}{*}{$SO(2n+1)_k$} & $(a_1, \ldots, a_n)$ & \\
                                & $(b_1, \ldots, b_n)$ & $a_n b_n \mod 2$ \\
\midrule
\multirow{2}{*}{$Sp(2n)_k$} & $(a_1, \ldots, a_n)$ & \\
                             & $(b_1, \ldots, b_n)$ & $\left(\sum_{j=1}^n ja_j \right)\left(\sum_{j=1}^n jb_j \right) \mod 2$ \\
\midrule
\multirow{2}{*}{$E_7(k)$} & $(a_1, \ldots, a_7)$ & \\
                            & $(b_1, \ldots, b_7)$ & $(a_2+a_5+a_7)(b_2+b_5+b_7) \mod 2$ \\
\midrule
\multirow{2}{*}{$SO(2n)_k, n=2s+1$} & $(a_1, \ldots, a_n)$ & \\
                                      & $(b_1, \ldots, b_n)$ & $\begin{array}{l}\left(\sum_{j=1}^{n-2}2a_j + (n-2) a_{n-1}+na_n\right) \\
                                      \times \left(\sum_{j=1}^{n-2}2b_j + (n-2) b_{n-1}+nb_n\right) \\ \mod 4 \end{array}$ \\
\bottomrule
\end{tabular}
\caption{Coefficient $\beta_{X,Y}$ in equation~\eqref{S-matrix cyclic general}}
\label{beta xy}
\end{table}
\subsubsection{Examples of cyclic zestings of Verlinde categories}

\begin{example}[Zestings of $SU(2)_k$, $k \equiv 2 \mod 4 $]
Following the notation of Example \ref{Example: modular data SU(2)}, let $k=2l$ where $l$ is odd. The simple object $f$, associated with $k$ in $\Irr(SU(2)_k)$, is a fermion; that is, it is an invertible object of order two and $\theta_{f} = -1$. The braided zestings correspond to pairs $(1,0), (1,1)$ and are given by powers of $(\lambda, \omega, t)$, defined as $\lambda(f,f) = f$, $\omega(f,f,f) = i$, and $t(f,f) = e^{\frac{\pi i}{4}}$. Then, the zested modular data are given by

\begin{align}
    s_{ab}^\lambda &= i^{-(ab \mod 2)}\sqrt{\frac{2}{k + 2}} \sin\left(\pi \frac{(a + 1)(b + 1)}{k + 2}\right), \\
    T_{aa}^\lambda &= e^{\frac{\pi i}{4}}  \exp \left( \frac{\pi i a^2}{2(k+2)} + \frac{\pi ia}{k+2} \right).
\end{align}

Note that the $S$-matrix of $SU(2)_k$ is real, while the zested $S^\lambda$-matrix is complex. This implies that in $SU(2)_k$, each simple object is self-dual. However, in the zested version of $SU(2)_k$ some objects are not longer self-dual. In particular, we have $N_{1,1}^0 = 1$ and $N_{1,1}^k = 0$, but under the zested fusion rules, $\tilde{N}_{1,1}^0 = 0$ and $\tilde{N}_{1,1}^k = 1$. Consequently, $X_1$ is not self-dual in the zested category; its dual is $X_{k-1}$.
\end{example}

\begin{example}[Case $E_8$]
It follows from \cite{sawin2006closed} that the Verlinde category $E_8(k)$ has a non-trivial invertible object only in the case that $k=2$. In this case, $E_8(2)$ is a braided Ising category, which are braided fusion categories with three simple objects

\begin{align*}
    \mathbf{0} &\leftrightarrow \unit &    \lambda_1 &\leftrightarrow \delta &      \lambda_8 &\leftrightarrow X,
\end{align*}
where $\delta \otimes \delta = \unit$, $X \otimes X = \unit \oplus \delta$, and $\delta \otimes X = X$. See \cite[Appendix B]{DGNO} for more details about Ising categories. The invertible object $\delta$ is a fermion, meaning $\theta_\delta = -1$. Since zesting in this case does not change the fusion rules, it follows from \cite[Theorem 3.15 (iii)]{fold-way} that the zesting construction acts transitively on the set of equivalence classes of braided Ising categories. Interestingly, $SU(2)_2$ is also an Ising category, corresponding to the complex conjugate of $E_8(2)$, meaning they are related by braided zesting.
\end{example}

\begin{example}\label{example:E6level3}
Consider $E_6(3)$, the Verlinde modular category associated to $E_6$ at level 3. $E_6(3)$ has 20 simple objects and it is graded by $\mathbb{Z}/3$. We associate the labels of simple objects with an $X_i$ as follows:
{\small 
\begin{align*}
    \mathbf{0} &\leftrightarrow X_0 &    \lambda_1 &\leftrightarrow X_1 &      \lambda_2 &\leftrightarrow X_2 &      \lambda_3 &\leftrightarrow X_3 \\
    \lambda_4 &\leftrightarrow X_4 &      \lambda_5 &\leftrightarrow X_5 &      \lambda_6 &\leftrightarrow X_6 &      2\lambda_1 &\leftrightarrow X_7 \\
    \lambda_1 + \lambda_2 &\leftrightarrow X_8 &  \lambda_1 + \lambda_3 &\leftrightarrow X_9 & \lambda_1 + \lambda_5 &\leftrightarrow X_{10} & \lambda_1 + \lambda_6 &\leftrightarrow X_{11} \\
    \lambda_2 + \lambda_6 &\leftrightarrow X_{12} & \lambda_3 + \lambda_6 &\leftrightarrow X_{13} & \lambda_5 + \lambda_6 &\leftrightarrow X_{14} & 2\lambda_6 &\leftrightarrow X_{15} \\
    3\lambda_1 &\leftrightarrow X_{16} & 2\lambda_1+\lambda_6 &\leftrightarrow X_{17} & \lambda_1+2\lambda_6 &\leftrightarrow X_{18} & 3\lambda_6 &\leftrightarrow X_{19},
\end{align*}
} 
here $\lambda_1,\ldots,\lambda_6$ are the fundamental weights presented in Table \ref{tab:root_system_E6} in the appendix.

The invertible objects are $\{X_0, X_{16}, X_{19}\}$ and the twist for all of them is 1, inducing the following $\mathbb{Z}/3$-grading: 
\begin{align*}
\text{Irr}(E_6(3))_0 &= \{X_{0}, X_{2}, X_{4}, X_{9}, X_{11}, X_{14}, X_{16}, X_{19}\}, \\
\text{Irr}(E_6(3))_1 &= \{X_1, X_5, X_8, X_{13}, X_{15}, X_{17}\}, \\
\text{Irr}(E_6(3))_2 &= \{X_3, X_6, X_7, X_{10}, X_{12}, X_{18}\}.
\end{align*}
The 9 associative zestings of $E_6(3)$ are parameterized by $(a, b) \in \mathbb{Z}/3 \times \mathbb{Z}/3$. The associative zesting functions are defined as follows:
\begin{align*}
\begin{array}{ll}
\lambda_a(i, j) = \begin{cases} 
1, & \text{if } i + j < 3, \\
X_{19}^{\otimes a}, & \text{if } i + j \geq 3,
\end{cases} &
\lambda_b(i, j, k) = \begin{cases} 
1, & \text{if } i + j < 3, \\
e^{\frac{2\pi i bk}{3}}, & \text{if } i + j \geq 3.
\end{cases}
\end{array}
\end{align*}
For braided zestings the relevant pairs are $(a, b) \in \{(0, 0),(1, 2),(2, 1)\}$. For $a=2$ the fusion rules change. In fact, for each simple object in $E_6(3)$, $\dim_{\mathbb C}\Hom (\mathbf{1}, X^{\otimes 3})\neq 0$, but zesting the fusion rules with $\lambda_2$, we find 
\[X_1\otimes_{\lambda_2}X_1\otimes_{\lambda_2}X_1= 2X_9 + X_4 + 2X_{14} + 3X_{11} + X_{19} + X_{16},\] leading to $\dim_{\mathbb C}\Hom (\mathbf{1}, X_1\otimes_{\lambda_2}X_1\otimes_{\lambda_2}X_1)= 0$. 

For the pair $(a,b)=(2,1)$, three braided zestings are obtained by choosing $\gamma$ such that $\gamma^3 = q^{-1}$, yielding new modular categories and modular data.
\end{example}

\begin{example}\label{example:so7level2}

Let us consider another example, $SO(7)_2$, corresponding to the Verlinde category associated to the root system $B_3$ at level $k=2$. This category has seven simple objects parameterized as follows:
\begin{align*}
\mathbf{0} \leftrightarrow X_0, \quad &\lambda_1 \leftrightarrow X_1, \quad \lambda_2 \leftrightarrow X_2, \quad \lambda_3 \leftrightarrow X_3, \\
2\lambda_1 \leftrightarrow X_4, \quad &\lambda_1+\lambda_3 \leftrightarrow X_5, \quad 2\lambda_3 \leftrightarrow X_6.
\end{align*}

The category $SO(7)_2$ is $\mathbb Z/2$-graded by 

\begin{align*}
\text{Irr}(SO(7)_2)_0 = \{X_0, X_1, X_2, X_4, X_6\}, &&
\text{Irr}(SO(7)_2)_1 = \{X_3, X_5\},
\end{align*}

and its fusion rules are presented in Table \ref{table fusion rules so(7) nivel 2}

{\small
\begin{table}[h!]
\begin{tabular}{|c||c|c|c|c|c|c|c|}
\hline
    & $\mathbf{0}$ & $\mathbf{1}$ & $\mathbf{2}$ & $\mathbf{3}$ & $\mathbf{4}$ & $\mathbf{5}$ & $\mathbf{6}$ \\
\hline\hline
$\mathbf{0}$ & $\mathbf{0}$ & $\mathbf{1}$ & $\mathbf{2}$ & $\mathbf{3}$ & $\mathbf{4}$ & $\mathbf{5}$ & $\mathbf{6}$ \\
\hline
$\mathbf{1}$ &  & $\mathbf{0} + \mathbf{2} + \mathbf{4}$ & $\mathbf{1} + \mathbf{6}$ & $\mathbf{3} + \mathbf{5}$ & $\mathbf{1}$ & $\mathbf{3} + \mathbf{5}$ & $\mathbf{2} + \mathbf{6}$ \\
\hline
$\mathbf{2}$ &  &  & $\mathbf{0} + \mathbf{6} + \mathbf{4}$ & $\mathbf{3} + \mathbf{5}$ & $\mathbf{2}$ & $\mathbf{3} + \mathbf{5}$ & $\mathbf{1} + \mathbf{2}$ \\
\hline
$\mathbf{3}$ &  &  &  & $\mathbf{0} + \mathbf{1} + \mathbf{2} + \mathbf{6}$ & $\mathbf{5}$ & $\mathbf{1} + \mathbf{2} + \mathbf{6} + \mathbf{4}$ & $\mathbf{3} + \mathbf{5}$ \\
\hline
$\mathbf{4}$ &  &  &  &  & $\mathbf{0}$ & $\mathbf{3}$ & $\mathbf{6}$ \\
\hline
$\mathbf{5}$ &  &  &  &  &  & $\mathbf{0} + \mathbf{1} + \mathbf{2} + \mathbf{6}$ & $\mathbf{3} + \mathbf{5}$ \\
\hline
$\mathbf{6}$ &  &  &  &  &  &  & $\mathbf{0} + \mathbf{1} + \mathbf{4}$ \\
\hline
\end{tabular}
\caption{Fusion rules of $SO(7)_2$}
\label{table fusion rules so(7) nivel 2}
\end{table}

}

This category admits four associative zestings, parameterized by $(a,b) \in \mathbb{Z}/2 \times \mathbb{Z}/2$. The fusion rules of the zested categories, in the case of $a=1$ are presented in Table \ref{tab: zsted so(7) nivel 2}. The new fusion category is not equivalent to the original one, as in $SO(7)_2$ all simple objects are self-dual. However, in the zested $SO(7)_2$, the objects $X_3$ and $X_5$ are duals of each other. Furthermore, this fusion category does not admit a braided structure and, therefore, is not a modular category. The new fusion category, which indeed admits a braided $\mathbb{Z}/2$-crossed structure, corresponds to the fusion rules $\text{FR}^{7,2}_{12}$ in \cite{vercleyen2023low}.

{\small 
\begin{table}[h!]
\begin{tabular}{|c||c|c|c|c|c|c|c|}
\hline
    & $\mathbf{0}$ & $\mathbf{1}$ & $\mathbf{2}$ & $\mathbf{3}$ & $\mathbf{4}$ & $\mathbf{5}$ & $\mathbf{6}$ \\
\hline\hline
$\mathbf{0}$ & $\mathbf{0}$ & $\mathbf{1}$ & $\mathbf{2}$ & $\mathbf{3}$ & $\mathbf{4}$ & $\mathbf{5}$ & $\mathbf{6}$ \\
\hline
$\mathbf{1}$ &  & $\mathbf{0} + \mathbf{2} + \mathbf{4}$ & $\mathbf{1} + \mathbf{6}$ & $\mathbf{3} + \mathbf{5}$ & $\mathbf{1}$ & $\mathbf{3} + \mathbf{5}$ & $\mathbf{2} + \mathbf{6}$ \\
\hline
$\mathbf{2}$ &  &  & $\mathbf{0} + \mathbf{6} + \mathbf{4}$ & $\mathbf{3} + \mathbf{5}$ & $\mathbf{2}$ & $\mathbf{3} + \mathbf{5}$ & $\mathbf{1} + \mathbf{2}$ \\
\hline
$\mathbf{3}$ &  &  &  & $\color{red}{\mathbf{1} + \mathbf{2} + \mathbf{6} + \mathbf{4}}$ & $\mathbf{5}$ & $\color{red}{\mathbf{0} + \mathbf{1} + \mathbf{2} + \mathbf{6}}$ & $\mathbf{3} + \mathbf{5}$ \\
\hline
$\mathbf{4}$ &  &  &  &  & $\mathbf{0}$ & $\mathbf{3}$ & $\mathbf{6}$ \\
\hline
$\mathbf{5}$ &  &  &  &  &  & $\color{red}{\mathbf{1} + \mathbf{2} + \mathbf{6} + \mathbf{4}}$ & $\mathbf{3} + \mathbf{5}$ \\
\hline
$\mathbf{6}$ &  &  &  &  &  &  & $\mathbf{0} + \mathbf{1} + \mathbf{4}$ \\
\hline
\end{tabular}
\caption{Fusion of the zested $SO(7)_2$}
\label{tab: zsted so(7) nivel 2}
\end{table}
}

\end{example}

\begin{example}\label{example:sp6level2}
The last example of this section is $Sp(6)_2$, Corresponding  to the Verlinde category associated to the root system $C_3$ at level $k=22$. This category has 10 simple objects parameterized as follows: 

\begin{align*}
\mathbf{0}\leftrightarrow X_0, &&\lambda_1 \leftrightarrow X_1, &&\lambda_2 \leftrightarrow X_2, &&\lambda_3 \leftrightarrow X_3, \\
2\lambda_1 \leftrightarrow X_4, &&\lambda_1 +\lambda_2 \leftrightarrow X_5,&& \lambda_1 +\lambda_3 \leftrightarrow X_6,\\ 2\lambda_2 \leftrightarrow X_7, &&\lambda_2 +\lambda_3 \leftrightarrow X_8, &&2\lambda_3 \leftrightarrow X_9
\end{align*}

The invertible objects are $\otimes$-generated by $X_9$ with $\theta_{X_9}=-1$ and $X_9^{\otimes 2}=X_0$. This defines the following $\mathbb Z/2$-grading

\begin{align*}
\text{Irr}(Sp(6)_2)_0=\{ X_0, X_2, X_4, X_6, X_7, X_9\}, &&    \text{Irr}(Sp(6)_2)_1=\{ X_1, X_3, X_5, X_8\}
\end{align*}

This category admits four associative zestings parametrized by $(a,b) \in \Z/2 \times \Z/2$. The associative zestings functions are defined by
\begin{align*}
\begin{array}{ll}
\lambda_a(i, j) = \begin{cases} 
1, & \text{if } i + j < 2, \\
X_{9}^{\otimes a}, & \text{if } i + j \geq 2,
\end{cases} &
\lambda_b(i, j, k) = \begin{cases} 
1, & \text{if } i + j < 2, \\
e^{\frac{2\pi k i(1+2b)}{4}}, & \text{if } i + j \geq 2.
\end{cases}
\end{array}
\end{align*}

The associative zestings corresponding to the pairs $(1,0)$ and $(1,1)$ have a braided structure given by the 2-cochain  $t(1,1)=\zeta$, where $\zeta^2=e^{\frac{2\pi k i(1+2b)}{4}}$, for $b=0$ and $b=1$ depending on whether we take the pair $(1,0)$ or $(1,1)$, respectively. The new fusion category is not equivalent to the original one, as in $Sp(6)_2$ every  simple object is self-dual. However, in the zested $Sp(6)_2$, the objects $X_1$ and $X_8$ are duals of each other.
\end{example}

\subsubsection{Verlinde Zesting of $SO(2n)_k$ with $n = 2s$}\label{subsubsection:so4s}

Among Verlinde modular categories, $SO(2n)_k$ with $n = 2s$ is the only family with non-cyclic group of invertible objects, see Table~\ref{table:center_of_groups}. 

As we have already observed in Table~\ref{Table:Admissible levels}, for $SO(2n)_k$ with $n = 2s$, the group of invertible objects $A$ is isomorphic to $\mathbb{Z}/2 \times \mathbb{Z}/2$. Additionally, $A_0 = A$ if $k$ is in $2\mathbb{Z}$, and $A_0 = \{1\}$ otherwise. 

For $k \in 2\mathbb{Z}$, the identification of $A$ with $\mathbb{Z}/2\times \mathbb{Z}/2$ and the twist coefficients restricted to $A$ are presented in Table \ref{twist_and_identifications}.

\begin{table}[h!]
\centering
\begin{tabular}{ll|ll}
\toprule
\multicolumn{2}{c|}{Twist coefficients} & \multicolumn{2}{c}{Identifications} \\
\midrule
$\theta_{\lambda}$ & $n=2s$ & $A$ & $\mathbb{Z}/2 \times \mathbb{Z}/2$ \\
\midrule
$\theta_{0}$ & $1$ & $0$ & $(0,0)$ \\
$\theta_{k\lambda_1}$ & $1$ & $k\lambda_1$ & $(1,1)$ \\
$\theta_{k\lambda_{n-1}}$ & $(-1)^{s}$ & $k\lambda_{n-1}$ & $(1,0)$ \\
$\theta_{k\lambda_{n}}$ & $(-1)^{s}$ & $k\lambda_{n}$ & $(0,1)$ \\
\bottomrule
\end{tabular}
\caption{Twist coefficients and identifications of invertible objects for $SO(2n)_k$ with $n$ and $k$ even.}
\label{twist_and_identifications}
\end{table}

Applying the results of Propositions \ref{Prop:Tannakian zesting Klein} and Theorem \ref{Braided zesting non tanakian klein}, we can characterize all the possible braided zestings of $SO(2n)_k$ with $k$ even as follows:

\begin{corollary}
\label{SO(2n) n even zestings}
The equivalence classes of braided zestings for $SO(2n)_k$ when $n$ and $k$ are even can be categorized into two families:
    \begin{itemize}
        \item If $n \in 4\Z$, the braided zestings are abelian 3-cocycles and can be parameterized according to Proposition \ref{prop:abelian3CocycleOfKleinGroup}.
        \item If $n \notin 4\Z$, braided zestings have been described and classified in Theorem \ref{Braided zesting non tanakian klein}.
    \end{itemize}
\end{corollary}
\begin{proof}
The conclusion follows directly from Table \ref{twist_and_identifications}.
\end{proof}

As a concluding result, we present the modular data for the braided zestings of $SO(2n)_k$ when $n$ is even and $n \notin 4\mathbb{Z}$.

\begin{theorem}
\label{zested_modular_data}
Consider the zested modular data of $SO(2n)_k$ for even $n$. These data are characterized by the equations:

\begin{itemize}
    \item If $n\in 4\Z$:
    \begin{align*}
        T_{X,X}^{\lambda} &= \zeta_4^{-a \left( \overline{\sum_{j=1}^{n-2}ja_j +a_n}\right)^2 - b\left(\overline{ \sum_{j=1}^{n-2}ja_j + a_{n-1}}\right)^2 } T_{X,X} \\
        s_{X,Y}^{\lambda} &= (-1)^{\left( a\left( \overline{\sum_{j=1}^{n-2}ja_j +a_n}\right)\left( \overline{\sum_{j=1}^{n-2}jb_j +b_n}\right)+ b\left( \overline{\sum_{j=1}^{n-2}ja_j+a_{n-1}}\right)\left( \overline{\sum_{j=1}^{n-2}jb_j+b_{n-1}}\right) \right)} \\
        &\quad \times (-1)^{c \left( \left( \overline{\sum_{j=1}^{n-2} ja_j + a_{n}} \right) \left( \overline{\sum_{j=1}^{n-2} jb_j + b_{n-1}} \right)+ \left( \overline{\sum_{j=1}^{n-2} ja_j + a_{n-1}} \right) \left( \overline{\sum_{j=1}^{n-2} jb_j + b_{n}} \right) \right)} s_{X,Y} 
    \end{align*}
 where $(a,b,c) \in \mathbb{Z}/4 \times \mathbb{Z}/4 \times \Z/2$ and $\zeta_4$ is a primitive fourth root of unity.

    \item If $n\notin 4\Z$:
    \begin{align*}
        T_{X,X}^{\lambda} &= \zeta_8^{-m_1 \left( \overline{\sum_{j=1}^{n-2} ja_j + a_{n-1}} \right)^2-m_2\left( \overline{\sum_{j=1}^{n-2} ja_j + a_n} \right)^2} (-1)^{c\left( \overline{\sum_{j=1}^{n-2} ja_j + a_{n-1}} \right) \left( \overline{\sum_{j=1}^{n-2} ja_j + a_n} \right)} T_{X,X}, \\
        s_{X,Y}^{\lambda} &= \zeta_{8}^{2 \left(m_1 \left( \overline{\sum_{j=1}^{n-2} ja_j + a_{n-1}} \right) \left( \overline{\sum_{j=1}^{n-2} jb_j + b_{n-1}} \right)+m_2 \left( \overline{\sum_{j=1}^{n-2} ja_j + a_n} \right) \left( \overline{\sum_{j=1}^{n-2} jb_j + b_n} \right) \right)} \\
        &\quad \times (-1)^{c \left( \left( \overline{\sum_{j=1}^{n-2} ja_j + a_{n-1}} \right) \left( \overline{\sum_{j=1}^{n-2} jb_j + b_n} \right)+ \left( \overline{\sum_{j=1}^{n-2} ja_j + a_n} \right) \left( \overline{\sum_{j=1}^{n-2} jb_j + b_{n-1}} \right) \right)} s_{X,Y}.
    \end{align*}
\end{itemize}
where $(m_1, m_2, c) \in \mathbb{Z}/4 \times \mathbb{Z}/4 \times \mathbb{Z}/2$, and $\zeta_8$ denotes a primitive eighth root of unity.
Here, $\overline{l}$ denotes the equivalence class of $l \mod 2$ for any integer $l \in \mathbb{Z}$. Also, $X, Y \in \text{Irr}(SO(2n)_k)$ are represented by the tuples $(a_1, \ldots, a_n)$ and $(b_1, \ldots, b_n)$, respectively.
\end{theorem}

\begin{proof}
The result follows directly from Theorem \ref{th:new_modular_data} through explicit calculations using the fixed $\mathbb{Z}/2 \times \mathbb{Z}/2$ grading.

For instance, for $n\notin 4\Z$ , since $n = 2s$, $s$ must be odd. According to Proposition \ref{grading fixed generators}, the degree of simple labels $X = (a_1, \ldots, a_n)$ and $Y = (b_1, \ldots, b_n)$ can be expressed as

\begin{align*}
       |X| &= \left( \overline{\sum_{j=1}^{n-2} j a_j + s a_{n-1} + (s-1) a_n}, \overline{\sum_{j=1}^{n-2} j a_j + (s-1) a_{n-1} + s a_n} \right) \\
       &= \left( \overline{\sum_{j=1}^{n-2} j a_j + a_{n-1}}, \overline{\sum_{j=1}^{n-2} j a_j + a_n} \right)
\end{align*}

With this, the conclusion follows from Theorem \ref{th:new_modular_data}.
\end{proof}

\begin{example}
As final example, we consider  $SO(12)_2$, corresponding to the Verlinde category associated to the root system $D_6$ at level $k=2$. This category has thirteen simple objects parameterized as follows:

\begin{align*}
\mathbf{0} \leftrightarrow X_0,  &&\lambda_1 \leftrightarrow X_1,  &&\lambda_2 \leftrightarrow X_2, \\
\lambda_3 \leftrightarrow X_3,  &&\lambda_4 \leftrightarrow X_4,  &&\lambda_5 \leftrightarrow X_5, \\
\lambda_6 \leftrightarrow X_6,  &&2\lambda_1 \leftrightarrow X_7,  &&\lambda_1+\lambda_5 \leftrightarrow X_8, \\
\lambda_1+\lambda_6 \leftrightarrow X_9,  &&2\lambda_5 \leftrightarrow X_{10},  &&\lambda_5+\lambda_6 \leftrightarrow X_{11}, \\
2\lambda_6 \leftrightarrow X_{12},
\end{align*}

The invertible objects are $X_0, X_7, X_{10}, X_{12}$ where $\theta_{X_{10}} = \theta_{X_{12}} = -1$, and then the category $SO(12)_2$ is $\mathbb{Z}/2 \times \mathbb{Z}/2$-graded by
\begin{align*}
\text{Irr}(SO(12)_2)_{(0,0)} &= \{ X_0, X_2, X_4, X_7, X_{10}, X_{12} \}, &&
\text{Irr}(SO(12)_2)_{(0,1)} = \{ X_6, X_8 \}, \\
\text{Irr}(SO(12)_2)_{(1,1)} &= \{ X_1, X_3, X_{11} \}, &&
\text{Irr}(SO(12)_2)_{(1,0)} = \{ X_5, X_9 \}.
\end{align*}

By Theorem \ref{Braided zesting non tanakian klein}, the braided zestings of $SO(12)_2$ are parametrized by $(a,b,c) \in \mathbb{Z}/4 \times \mathbb{Z}/4 \times \mathbb{Z}/2$, where the associated 2-cocycle to $(a,b,c)$ is given by $$\lambda_{a \bmod 2, b \bmod 2}(\vec{x},\vec{y}) = (ax_1y_1 \bmod 2, bx_2y_2 \bmod 2).$$ The triples $(a,b,c)$ where $a = b = 0 \bmod 2$ correspond to abelian 3-cocycles, and for them, the fusion rules do not change. For the 2-cocycles $\lambda_{0,1}$, $\lambda_{1,0}$, and $\lambda_{1,1}$, we obtain new modular categories, where the fusion rules change. In fact, every simple objects in $SO(12)_2$ is self-dual, however as we can see in Table \ref{tab:dualities_2-cocycles}, the dualities change depending on the 2-cocycle.

\begin{table}[h!]
\centering
\begin{tabular}{ll|ll}
\toprule
\multicolumn{2}{c|}{2-Cocycle} & \multicolumn{2}{c}{New Dualities} \\
\midrule
\multicolumn{2}{l|}{$\lambda_{1,1}$} & \multicolumn{2}{l}{$X_5^* = X_9$, $X_6^* = X_8$} \\
\multicolumn{2}{l|}{$\lambda_{1,0}$} & \multicolumn{2}{l}{$X_1^* = X_{11}$, $X_5^* = X_9$} \\
\multicolumn{2}{l|}{$\lambda_{0,1}$} & \multicolumn{2}{l}{$X_1^* = X_{11}$, $X_6^* = X_8$} \\
\bottomrule
\end{tabular}
\caption{Dualities in the zested $SO(12)_2$ under different 2-cocycles}
\label{tab:dualities_2-cocycles}
\end{table}
\end{example}

\begin{remark}
We have classified and constructed all possible braiding zesting data for the universal grading group of Verlinde categories. Since for a Verlinde category, most of its universal grading groups are $\mathbb{Z}/2$ and $\mathbb{Z}/3$, not much is lost by considering universal grading groups, with the exception of the cases $SU(N+1)_k$ and $SO(2N)_k$ with even $N$, which have gradings $\mathbb{Z}/N$ and $\mathbb{Z}/2 \times \mathbb{Z}/2$, respectively. Considering different gradings corresponds to taking subgroups of the universal group, and in this case, some additional complications may arise in the $H^4$-obstruction. The associated zesting data will now take values in a pointed braided category, which is not necessarily symmetric, and therefore the definition of associative and braided zesting for a pre-metric group needs to be slightly modified. Note that the general theory was developed in detail in \cite{zesting}; it is just that given the problem we wanted to address in this paper, our definition and results were restricted to the universally graded case.
\end{remark}

\section*{Appendix}
On this appendix, we will recall some of the formulas from root systems that we need in the computations of the paper, following  \cite{bourbaki2002lie} and \cite{humphreys2012introduction}.

\begin{table}[h!]
\centering
\begin{tabular}{ll}
\toprule
Root System $A_n$ \, \, $n\geq1$ & $E\subset \mathbb{R}^{n+1} $ Orthogonal subspace to $e_1+\cdots+ e_{n+1}$        \\ \midrule
Basis                           & $\alpha_1=e_1-e_2, \cdots, \alpha_{n}=e_n-e_{n+1}$                               \\ 
Normalized inner product & $\left\langle \ , \ \right \rangle=( \, , \, )$ \\ 
Dual Coxeter                    & $h^\vee=n+1$                                                                     \\ 
Highest root                    & $\Theta=e_1-e_{n+1}=\alpha_1+ \cdots+ \alpha_{n}$                                \\ 
Fundamental Weights             & $\displaystyle{\lambda_i= (e_1+\cdots+e_i)- \frac{i}{n+1}\sum_{j=1}^{n+1}e_{j}}$ \\ 
Sum of Positive Roots          & $2\rho=ne_1+(n-2)e_2+(n-4)e_3+\cdots+(n-2n)e_{n+1}$                              \\ \bottomrule

\end{tabular}
\caption{}
\label{tab:root_system_An}
\end{table}

\begin{table}[h!]
\centering
\begin{tabular}{ll}
\toprule
Root System $B_n$ \, \, $n\geq2$ & $E=\mathbb{R}^{n} $                                                                                                                                                                                                                                                                                                                                                                           \\ \midrule
Basis                           & $\alpha_1=e_2-e_3, \cdots, \alpha_{n-2}=e_{n-1}-e_{n}, \alpha_{n-1}=e_{n-1}+e_n$, $\alpha_{n}=e_n$                                                                                                                                                                                                                                                                                                                     \\ 
Normalized inner product & $\left\langle \ , \ \right \rangle=( \, , \, )$ \\ 
Dual Coxeter                    & $h^\vee=2n-1$                                                                                                                                                                                                                                                                                                                                                                                 \\ 
Highest root                    & $\Theta=e_2+e_{n}=\alpha_1+ 2\alpha_2+2\alpha_3+ \cdots+ 2\alpha_{n-1}+ \alpha_n$                                                                                                                                                                                                                                                                                                                         \\ 
Fundamental Weights             & 
$\begin{array}{cc}
\lambda_i=e_2 + e_3 +\cdots+ e_{i+1}  &  i < n-1 \\
  \lambda_{n-1} =\frac{1}{2}(e_2+ \cdots+ e_{n-1}+ e_n) \\
  \lambda_n =\frac{1}{2}(e_2+ \cdots+ e_{n-1}+ e_n)  & 
\end{array}$

 \\ 
Sum of Positive Roots          & $2\rho=(2n-1)e_2+(2n-3)e_3+\cdots+(2n-(2n-1))e_{n}$                                                                                                                                                                                                                                                                                                                                           \\ \bottomrule
\end{tabular}
\caption{}
\label{tab:root_system_Bn}
\end{table}

\begin{table}[h!]
\centering
\begin{tabular}{ll}
\toprule
Root System $C_n$ \, \, $n\geq2$ & $E=\mathbb{R}^{n} $                                                          \\ \midrule
Basis                           & $\alpha_1=e_1-e_2, \cdots, \alpha_{n-1}=e_{n-1}-e_{n}$, $\alpha_{n}=2e_n$.   \\ 
Normalized inner product & $\left\langle \ , \ \right \rangle=( \, , \, )/2$ \\ 
Dual Coxeter                    & $h^\vee=n+1$                                                                 \\ 
Highest root                    & $\Theta=2e_1=2\alpha_1+ 2\alpha_2+2\alpha_3+ \cdots+ 2\alpha_{n-1}+\alpha_n$ \\ 
Fundamental Weights             & $\lambda_i=e_1+e_2+\cdots+ e_i$                                              \\ 
Sum of Positive Roots          & $2\rho=2ne_1+(2n-2)e_2+\cdots+(2n-2(n-1))e_{n}$                              \\ \bottomrule
\end{tabular}
\caption{}
\label{tab:root_system_Cn}
\end{table}

\begin{table}[h!]
\centering
\begin{tabular}{ll}
\toprule
Root System $D_n$ \, \, $n\geq3$ & $E=\mathbb{R}^{n} $                                                                                                                                                                                                                          \\ \midrule
Basis                  & $\alpha_1=e_1-e_2, \cdots, \alpha_{n-1}=e_{n-1}-e_{n}$, $\alpha_{n}=e_{n-1}+e_n$.                                                                                                                                                            \\ 
Normalized inner product & $\left\langle \ , \ \right \rangle=( \, , \, )$ \\ 
Dual Coxeter           & $h^\vee=2n-2$                                                                                                                                                                                                                                \\ 
Highest root           & $\Theta=e_1+e_2=\alpha_1+ 2\alpha_2+ \cdots+ 2\alpha_{n-2}+\alpha_{n-1}+\alpha_n$                                                                                                                                                            \\ 
Fundamental Weights    & $\begin{array}{cc}\lambda_i=e_1 + e_2 +\cdots+ e_{i}  &  i \leq n-2 \\   \lambda_{n-1} =\frac{1}{2}(e_1+ \cdots+ e_{n-1}-e_n)  &  \\ \lambda_{n}=\frac{1}{2}(e_1+\cdots e_n) &\\ \end{array}$ \\ 
Sum of Positive Roots & $2\rho=2(n-1)e_1+2(n-2)e_2+\cdots+2(n-(n-1))e_{n-1}$                                                                                                                                                                                         \\ \bottomrule
\end{tabular}
\caption{}
\label{tab:root_system_Dn}
\end{table}

\begin{table}[h!]
\centering
\begin{tabular}{ll}
\toprule
Root System $E_6$ & $E\subset \mathbb{R}^{8} $                                                                                                                                                                                                                                                                                                                                                                  \\ \midrule
Basis                  & \scriptsize{\begin{tabular}[c]{@{}l@{}}$\begin{array}{cc} \alpha_1=\frac{1}{2}(e_1+e_8) -\frac{1}{2}(e_2+\cdots+e_7)  &  \alpha_2=e_1+e_2, \cdots ,\alpha_6= e_5-e_4\\

\end{array}$\end{tabular}}                                                                                                                                                                                                                                                                                                         \\ 
Normalized inner product & $\left\langle \ , \ \right \rangle=( \, , \, )$ \\ 
Dual Coxeter           & $h^\vee=12$                                                                                                                                                                                                                                                                                                                                                                                 \\ 
Highest root           & \scriptsize{$\Theta=\frac{1}{2}(e_1+e_2+\cdots+e_5-e_6-e_7+e_8)=\alpha_1+ 2\alpha_2+ 2\alpha_3+3\alpha_4+2\alpha_5+\alpha_6$ }                                                                                                                                                                                                                                                                           \\ 
Fundamental Weights    & \scriptsize{\begin{tabular}[c]{@{}l@{}}$\begin{array}{cc}\\ \lambda_1=\frac{2}{3}(e_8-e_7-e_6)   &  \lambda_2=\frac{1}{2}(e_1+\cdots+e_5-e_6-e_7+e_8) \\\\   \lambda_{3} =\frac{5}{6}(e_8-e_7-e_6)+ \frac{1}{2}(-e_1+e_2+\cdots+e_5)  &  \lambda_{4}=e_3+e_4+e_5-e_6-e_7+e_8 \\\\ \lambda_{5}=\frac{2}{3}(e_8-e_7-e_6)+e_4 + e_5 &\lambda_{6}=\frac{1}{3} (e_8-e_7-e_6)+e_5\\ \end{array}$\end{tabular}} \\ 
Sum of Positive Roots & \scriptsize{$2\rho=2( e_2+2e_3+3e_4+4e_5+4(e_8-e_7-e_6)$   }                                                                                                                                                                                                                                                                                                                                             \\ \bottomrule
\end{tabular}
\caption{}
\label{tab:root_system_E6}
\end{table}

\begin{table}[h!]
\centering
\begin{tabular}{ll}
\toprule
Root System $E_7$ & $E\subset \mathbb{R}^{8} $  \small{  Orthogonal to} $e_7+e_8$                                                                                                                                                                                                                                                                                                                                                              \\ \midrule
Basis                  & \scriptsize{\begin{tabular}[c]{@{}l@{}}$\begin{array}{cc} \alpha_1=\frac{1}{2}(e_1+e_8) -\frac{1}{2}(e_2+\cdots+e_7)  &  \alpha_2=e_1+e_2, \cdots ,\alpha_7= e_6-e_5\\

\end{array}$\end{tabular}}     \\ 
Normalized inner product & $\left\langle \ , \ \right \rangle=( \, , \, )$                                                                                                                                                                                                                                                                                                      \\ 
Dual Coxeter           & $h^\vee=18$                                                                                                                                                                                                                                                                                                                                                                                 \\ 
Highest root           & \scriptsize{$\Theta=e_8-e_7=2\alpha_1+ 2\alpha_2+ 3\alpha_3+4\alpha_4+3\alpha_5+2\alpha_6+ \alpha 7$ }                                                                                                                                                                                                                                                                           \\ 
Fundamental Weights    & \scriptsize{\begin{tabular}[c]{@{}l@{}}$\begin{array}{cc}\\ \lambda_1= e_8-e_7 &  \lambda_2=\frac{1}{2}(e_1+\cdots+e_6-2e_7-+2e_8) \\\\   \lambda_{3} =\frac{1}{2}(-e_1+e_2+\cdots+e_6-3e_7+3e_8)  &  \lambda_{4}=e_3+\cdots+e_6+2(e_8-e_7) \\\\ \lambda_{5}=\frac{1}{2}(2e_4+2e_5+2e_6+3(e_8-e_7)) &\lambda_{6}=e_5+e_6-e_7+e_8\\\\
\lambda_{7}=e_6 + \frac{1}{2}(e_8-e_7) & \\
\end{array}$\end{tabular}} \\ 
Sum of Positive Roots & \scriptsize{$2\rho=2( e_2+2e_3+3e_4+4e_5+4(e_8-e_7-e_6)$   }                                                                     \\ \bottomrule
\end{tabular}
\caption{}
\label{tab:root_system_E7}
\end{table}

\section*{Data Availability Statement}
No data were used to support the findings of this study. The research is purely theoretical and does not rely on any empirical data sets.


\begin{thebibliography}{NRWW23}

\bibitem[BGH{\etalchar{+}}14]{AIM2012}
Paul Bruillard, C\'{e}sar Galindo, Seung-Moon Hong, Yevgenia Kashina, Deepak Naidu, Sonia Natale, Julia~Yael Plavnik, and Eric~C. Rowell.
\newblock Classification of integral modular categories of {F}robenius-{P}erron dimension {$pq^4$} and {$p^2q^2$}.
\newblock {\em Canad. Math. Bull.}, 57(4):721--734, 2014.

\bibitem[BGH{\etalchar{+}}17]{fold-way}
Paul Bruillard, C\'esar Galindo, Tobias Hagge, Siu-Hung Ng, Julia~Yael Plavnik, Eric~C. Rowell, and Zhenghan Wang.
\newblock Fermionic modular categories and the 16-fold way.
\newblock {\em J. Math. Phys.}, 58(4):041704, 31, 2017.

\bibitem[BK01]{BK}
Bojko Bakalov and Alexander Kirillov, Jr.
\newblock {\em Lectures on tensor categories and modular functors}, volume~21 of {\em University Lecture Series}.
\newblock American Mathematical Society, Providence, RI, 2001.

\bibitem[Bou02]{bourbaki2002lie}
Nicolas Bourbaki.
\newblock Lie groups and lie algebras. chapters 4--6. translated from the 1968 french original, 2002.

\bibitem[DGNO10]{DGNO}
Vladimir Drinfeld, Shlomo Gelaki, Dmitri Nikshych, and Victor Ostrik.
\newblock On braided fusion categories. {I}.
\newblock {\em Selecta Math. (N.S.)}, 16(1):1--119, 2010.

\bibitem[DGP{\etalchar{+}}21]{zesting}
Colleen Delaney, C\'{e}sar Galindo, Julia Plavnik, Eric~C. Rowell, and Qing Zhang.
\newblock Braided zesting and its applications.
\newblock {\em Comm. Math. Phys.}, 386(1):1--55, 2021.

\bibitem[DN21]{davydov2021braided}
Alexei Davydov and Dmitri Nikshych.
\newblock Braided {P}icard groups and graded extensions of braided tensor categories.
\newblock {\em Selecta Math. (N.S.)}, 27(4):Paper No. 65, 87, 2021.

\bibitem[EGNO15]{EGNO}
Pavel Etingof, Shlomo Gelaki, Dmitri Nikshych, and Victor Ostrik.
\newblock {\em Tensor categories}, volume 205 of {\em Mathematical Surveys and Monographs}.
\newblock American Mathematical Society, Providence, RI, 2015.

\bibitem[EML54]{MR0065162}
Samuel Eilenberg and Saunders Mac~Lane.
\newblock On the groups {$H(\Pi,n)$}. {II}. {M}ethods of computation.
\newblock {\em Ann. of Math. (2)}, 60:49--139, 1954.

\bibitem[ENO05]{ENO}
Pavel Etingof, Dmitri Nikshych, and Viktor Ostrik.
\newblock On fusion categories.
\newblock {\em Ann. of Math. (2)}, 162(2):581--642, 2005.

\bibitem[Fuc91]{fuchs1991simple}
J\"{u}rgen Fuchs.
\newblock Simple {WZW} currents.
\newblock {\em Comm. Math. Phys.}, 136(2):345--356, 1991.

\bibitem[GJ16]{Galindo-Jaramillo}
C\'{e}sar Galindo and Nicol\'{a}s Jaramillo.
\newblock Solutions of the hexagon equation for abelian anyons.
\newblock {\em Rev. Colombiana Mat.}, 50(2):273--294, 2016.

\bibitem[GN08]{nilpotent}
Shlomo Gelaki and Dmitri Nikshych.
\newblock Nilpotent fusion categories.
\newblock {\em Adv. Math.}, 217(3):1053--1071, 2008.

\bibitem[HLY14]{huang2014braided}
Hua-Lin Huang, Gongxiang Liu, and Yu~Ye.
\newblock The braided monoidal structures on a class of linear gr-categories.
\newblock {\em Algebras and Representation Theory}, 17(4):1249--1265, 2014.

\bibitem[Hum78]{humphreys2012introduction}
James~E. Humphreys.
\newblock {\em Introduction to {L}ie algebras and representation theory}, volume~9 of {\em Graduate Texts in Mathematics}.
\newblock Springer-Verlag, New York-Berlin, 1978.
\newblock Second printing, revised.

\bibitem[Kar85]{karpilovsky1985projective}
Gregory Karpilovsky.
\newblock {\em Projective representations of finite groups}, volume~94 of {\em Monographs and Textbooks in Pure and Applied Mathematics}.
\newblock Marcel Dekker, Inc., New York, 1985.

\bibitem[Kar89]{KARPY}
Gregory Karpilovsky.
\newblock {\em Clifford theory for group representations}, volume 156 of {\em North-Holland Mathematics Studies}.
\newblock North-Holland Publishing Co., Amsterdam, 1989.
\newblock Notas de Matem{\'a}tica [Mathematical Notes], 125.

\bibitem[M\"03]{MR1990929}
Michael M\"{u}ger.
\newblock On the structure of modular categories.
\newblock {\em Proc. London Math. Soc. (3)}, 87(2):291--308, 2003.

\bibitem[NRW23]{NRW}
Siu-Hung Ng, Eric~C. Rowell, and Xiao-Gang Wen.
\newblock Classification of modular data up to rank 11.
\newblock {\em \emph{Preprint:} arXic/2308.09670}, 2023.

\bibitem[NRWW23]{NRWW}
Siu-Hung Ng, Eric~C. Rowell, Zhenghan Wang, and Xiao-Gang Wen.
\newblock Reconstruction of modular data from {${\rm SL}_2(\Bbb Z)$} representations.
\newblock {\em Comm. Math. Phys.}, 402(3):2465--2545, 2023.

\bibitem[RSW09]{RSW}
Eric Rowell, Richard Stong, and Zhenghan Wang.
\newblock On classification of modular tensor categories.
\newblock {\em Comm. Math. Phys.}, 292(2):343--389, 2009.

\bibitem[RW18]{RW}
Eric~C. Rowell and Zhenghan Wang.
\newblock Mathematics of topological quantum computing.
\newblock {\em Bull. Amer. Math. Soc. (N.S.)}, 55(2):183--238, 2018.

\bibitem[Saw02]{sawin2002jones}
Stephen~F. Sawin.
\newblock Jones-{W}itten invariants for nonsimply connected {L}ie groups and the geometry of the {W}eyl alcove.
\newblock {\em Adv. Math.}, 165(1):1--34, 2002.

\bibitem[Saw06]{sawin2006closed}
Stephen~F. Sawin.
\newblock Closed subsets of the {W}eyl alcove and {TQFT}s.
\newblock {\em Pacific J. Math.}, 228(2):305--324, 2006.

\bibitem[VS23]{vercleyen2023low}
Gert Vercleyen and Joost Slingerland.
\newblock On low rank fusion rings.
\newblock {\em \emph{Preprint}: arXiv/2205.15637}, 2023.

\end{thebibliography}
\newcommand{\etalchar}[1]{$^{#1}$}

\end{document}